\documentclass[twoside,11pt,preprint]{article}

\usepackage{jmlr2e}
\usepackage{enumitem}
\usepackage{amsmath, amsfonts, amssymb}
\usepackage{eucal, xcolor}

\newcommand\bb\mathbb
\newcommand\bs\boldsymbol

\newcommand{\ELBO}{\operatorname{ELBO}}

\newcommand{\tr}{\operatorname{tr}}

\newcommand{\eps}{\epsilon}
\newcommand{\ip}[1]{\langle #1 \rangle}

\newcommand{\KL}{D_{\mathrm{KL}}}

\newcommand{\ff}{{\bs{f\!f}}}

\newcommand{\T}{{\mathrm T}}

\newcommand{\lambdavar}{\hat\lambda_n^{\mathrm{var}}}

\usepackage{lastpage}
\jmlrheading{26}{2025}{1-39}{4/25; Revised }{}{25-0000}{Dennis Nieman and Botond Szabó}

\ShortHeadings{}{Nieman and Szabó}
\firstpageno{1}

\begin{document}

\title{Adaptive sparse variational approximations for Gaussian process regression}

\author{\name Dennis Nieman \email dennis.nieman@hu-berlin.de \\
	\addr Institut für Mathematik \\
	Humboldt-Universität zu Berlin \\
    \AND
    \name Botond Szabó \email botond.szabo@unibocconi.it \\
    \addr Department of Decision Sciences, \\
    Bocconi Institute for Data Science and Analytics, \\
    Bocconi University, Milan
}

\editor{\ldots}

\maketitle

\begin{abstract}
Accurate tuning of hyperparameters is crucial to ensure that models can generalise effectively across different settings. In this paper, we present theoretical guarantees for hyperparameter selection using variational Bayes in the nonparametric regression model. We construct a variational approximation to a hierarchical Bayes procedure, and derive upper bounds for the contraction rate of the variational posterior in an abstract setting. The theory is applied to various Gaussian process priors and variational classes, resulting in minimax optimal rates. Our theoretical results are accompanied with numerical analysis both on synthetic and real world data sets.
\end{abstract}


\section{Introduction}

A core challenge in Bayesian statistics is scalability, i.e. the computation of the posterior for large sample sizes. Variational Bayes approximation is a standard approach to speed up inference. Variational posteriors are random probability measures that minimise the Kullback-Leibler divergence between a suitable class of distributions and the otherwise hard to compute posterior. Typically, the variational class of distributions over which the optimisation takes place does not contain the original posterior, hence the variational procedure can be viewed as a projection onto this class. The projected variational distribution then approximates the posterior. During the approximation procedure one inevitably loses information and hence it is important to characterize the accuracy of the approach. Despite the wide use of variational approximations, their theoretical underpinning started to emerge only recently, see for instance \cite{alquier2020,yang2020,zhang2020,ray2022}.


	In a Bayesian procedure, the choice of prior reflects the presumed properties of the unknown parameter. In comparison to regular parametric models, where in view of the Bernstein-von Mises theorem the posterior is asymptotically normal, the prior plays a crucial role in the asymptotic behaviour of the posterior. In fact, the large-sample behaviour of the posterior typically depends intricately on the choice of prior hyperparameters, so it is vital that these are tuned correctly. Many Bayesian procedures for model selection/hyperparameter tuning have been proposed and studied over the years. The two classical approaches are hierarchical and empirical Bayes methods. In hierarchical Bayes the tuning hyperparameters are endowed with another layer of prior resulting in a multi-layer, hierarchical prior distribution, while in empirical Bayes the hyperparameters are estimated empirically and plugged into the posterior; see for instance the monograph \cite{ghosal2017} for an overview of these methods.  However, these classical methods, especially in more complex models and large data sizes, can be numerically very slow. In this paper we propose a method for variational hyperparameter tuning that speeds up the computations and provides reliable inference, as supported both by theory and empirical evidence. 

Although our approach can be more generally applied, we mainly focus on sparse variational approximations for Gaussian process (GP) posteriors in the nonparametric regression model. We consider as examples both the inducing variable approach proposed by \cite{titsias2009a} and a standard mean-field variational approach for approximating the posterior. This article extends a line of research in \cite{vbgp, vbuq}, where minimax convergence rates and reliable uncertainty quantification were derived for the variational posteriors. In these papers, however, it was assumed that the regularity of the underlying, data-generating function is known and the optimal tuning of the procedures heavily depended on it. In practice such information is usually not available, so here, in contrast, we modify the approximation scheme in such a way that the knowledge of the true regularity is no longer required. Hence, our procedure provides data-driven, adaptive inference on the unknown regularity classes.

\paragraph{Related literature.} Adaptive contraction rates for variational methods were first studied for tempered posteriors, where in Bayes' rule the likelihood is raised to a power $\alpha< 1$, diminishing its importance. In
 \cite{cherief2018} convergence rates were derived for mixture models, discussing also model selection using a regularised evidence lower bound (ELBO) criterion.  In subsequent work, \cite{cherief2019} takes a similar approach in a more general setting with discrete model selection parameter. For standard posterior distributions \cite{zhang2020} have derived optimal contraction rate guarantees for mean-field variational approximation of hierarchical sieve priors in context of the many normal means model and later extended the results to more general high-dimensional settings in \cite{zhang2020a}. Recently, an adaptive variational approach was proposed in \cite{ohn2024}, by considering a mixture of variational approximations of the individual models. However, none of the present papers addresses the Gaussian process regression model and the investigated variational algorithms are also different than the ones we focus on.

\paragraph{Contributions.} Below is a short overview of our results.
\begin{itemize}[label=--]
\item We start by comparing various variational approaches for data driven tuning of the hyperparameter.
\item Then we derive contraction rate results for general (discrete mixture) hierarchical priors in context of the nonparametric regression model and show that the proposed variational approximation inherits the contraction rate under a condition on the Kullback-Leibler divergence between these measures.
\item We apply these general results for Gaussian process priors using two different type of variational classes, i.e. inducing variable methods and a mean-field type approximation. We derive minimax rate adaptive contraction over a range of Sobolev-type regularity classes.
\item Our theoretical results are accompanied with numerical analysis, considering both synthetic and real world data sets.
\end{itemize}

\paragraph{Organization of the paper.} 
In Section~\ref{s:vb.modelselection} we review variational Bayes methods and introduce a variational approach that we later show to be adaptive on the regularity of the truth. The notation is tailored to the nonparametric regression setting but the framework is general. Next, in Section~\ref{c:adaptive.rates}, general $L^2$ contraction rate theorems on hierarchical and variational posteriors in the nonparametric regression model are given. The theory is applied in concrete examples in Section~\ref{s:ex}, including different variational methods and priors. Beside theoretical guarantees also numerical results are provided. All proofs and supporting lemmas are given in the subsequent appendix. 

\section{Model selection with variational Bayes}\label{s:vb.modelselection}

We consider the nonparametric regression model, where the goal is to infer a function $f_0 : \mathcal X \to \bb R$ from data
\begin{equation}\label{e:np.reg}
y_i = f_0(x_i) + \varepsilon_i, \qquad i=1,\ldots,n. 
\end{equation}
The domain $\mathcal X$ is a subset of $\bb R^d$, and we assume \emph{random design}, meaning that $x_1,\ldots,x_n$ are i.i.d. from some distribution $G$ on $\mathcal X$. In our theory the common variance $\sigma^2$ of the errors $\varepsilon_1,\ldots,\varepsilon_n \sim_{\text{iid}} \mathcal N(0,\sigma^2)$ is assumed known, but later on we also demonstrate empirically that it can be estimated variationally.

In the Bayesian framework one endows the functional parameter with a prior distribution. We will consider Gaussian process priors constructed using a covariance kernel $k : \mathcal X^2 \to \bb R$. Such priors $\Pi_\lambda$ typically rely on a collection of scaling or tuning hyperparameters $\lambda$, which, especially in high-dimensional and nonparametric settings, substantially influence the behaviour of the corresponding posterior. The optimal choices of these hyperparameters rely on the characteristics of the underlying truth $f_0$. Since this information is typically not available, in practice data-driven choices are considered. A fully Bayesian approach is to endow the hyperparameter $\lambda$ with a prior $\pi$, resulting in a hierarchical prior in the form 
\[ \Pi(\cdot) = \int \Pi_\lambda(\cdot) \, d\pi(\lambda), \]
where the hyper-prior $\pi$ can be a density or a probability mass function. The prior $\Pi$ is also called a \emph{mixture} (in our case a mixture of Gaussian processes). The hierarchical posterior $\Pi(\,\cdot\,| \bs x,\bs y)$ corresponding to this prior is given by Bayes' rule
\begin{equation}\label{e:BayesRule}
\frac{d\Pi(\,\cdot\,|\bs x,\bs y)}{d\Pi}(f) = \frac{p_f^{(n)}(\bs x,\bs y)}{\int p_g^{(n)}(\bs x,\bs y) \,d\Pi(g)},
\end{equation}
where $p_f^{(n)}$ denotes the Gaussian product likelihood in the nonparametric regression model, and $\bs x,\bs y$ are the respective vectors of observations in \eqref{e:np.reg}. Along the same lines we denote by $\Pi_\lambda(\,\cdot\,| \bs x,\bs y)$ the posterior corresponding to the prior $\Pi_\lambda$. Such hierarchical Bayesian procedures are widely used in the literature, however, especially for complex, high-dimensional models they face computational challenges. This is also the case in hierarchical Gaussian process regression. Hence in order to scale up the procedure the posterior is approximated by a distribution that is easier to compute.

In the variational approach the first step is to set the variational class, on which the posterior is projected. Let us consider a collection of variational classes $\mathcal Q_\lambda$ indexed by the same hyperparameter $\lambda$. The variational approximation $\tilde\Pi_\lambda(\,\cdot\,|\bs x,\bs y)$ is then defined by projecting the hierarchical posterior onto the class $\mathcal Q_\lambda$, i.e.
\begin{align*}
\tilde\Pi_\lambda(\,\cdot\,|\bs x,\bs y)=\arg\min_{Q\in \mathcal Q_\lambda}\KL(Q \| \Pi(\,\cdot\,| \bs x,\bs y)),
\end{align*}
where $\KL(Q\|P) = \int \log (\frac{dQ}{dP}) \,dQ$ whenever the integral is well-defined. Recall that using \eqref{e:BayesRule} the KL-divergence can be rewritten in the form
\[ 
	\KL(Q \| \Pi(\,\cdot\,|\bs x,\bs y)) = \log \int p_f^{(n)}(\bs x,\bs y) \,d\Pi(f) - \ELBO(Q),
\]
where the `evidence lower bound' function is defined as
\begin{equation}\label{4:e:ELBO}
\ELBO(Q) := \int \log p_f^{(n)}(\bs x,\bs y) \,dQ(f) - \KL(Q \| \Pi).
\end{equation}
Hence, as is well known, minimizing the KL-divergence is equivalent with maximizing the ELBO function. For each hyperparameter value $\lambda$, this results in an approximation $\tilde\Pi_\lambda(\,\cdot\,|\bs x,\bs y)$ of the hierarchical posterior $\Pi(\,\cdot\,|\bs x,\bs y)$. We choose among these approximations the one which minimizes the KL-divergence or equivalently maximizes the ELBO, i.e., we estimate $\lambda$ as maximiser of
\begin{equation}
\lambda \mapsto \ELBO(\tilde\Pi_\lambda(\,\cdot\,|\bs x,\bs y)) = \max_{Q \in \mathcal Q_\lambda} \ELBO(Q).\label{def:ELBO:func}
\end{equation}

For computational reasons and to keep the presentation of the results clean, we optimize the hyperparameter $\lambda$ only over a discrete, finite index set $\Lambda_n$. As is made explicit in the notation, we allow for this set to change or grow with the number of observations. Our results may be extended to the case of non-discrete hyperparameters, e.g. continuous hyper-priors $\pi$, using the idea of likelihood transformation in \cite{donnet2018} and \cite{rousseau2017}. 

To summarise, the variational approximation to the hierarchical posterior is the distribution
\begin{equation}\label{def:vb:posterior}
\tilde\Pi_{\lambdavar}(\,\cdot\, | \bs x,\bs y) = \tilde\Pi_\lambda(\,\cdot\, | \bs x,\bs y)\Big|_{\lambda=\lambdavar} = \arg\max\nolimits_{Q \in \bigcup_\lambda \mathcal Q_\lambda} \ELBO(Q),
\end{equation}
where
\begin{equation}\label{e:lambda.hat.var}
	\lambdavar
	= \arg\max_{\lambda\in\Lambda_n} \ELBO(\tilde\Pi_\lambda(\,\cdot\, | \bs x,\bs y)).
\end{equation}

An alternative objective function, proposed in e.g. \cite{titsias2009a} and \cite{hensman2013}, is the evidence lower bound that does not involve the hierarchical prior $\Pi$, but $\Pi_\lambda$ with fixed hyperparameter $\lambda$,
\[ \ELBO_\lambda(Q) := \int \log p_f^{(n)}(\bs x,\bs y) \,dQ(f) - \KL(Q \| \Pi_\lambda). \]
Maximizing the above function over $Q \in \mathcal Q_\lambda$ is equivalent to minimizing the Kullback-Leibler divergence between the variational class $\mathcal Q_\lambda$ and the posterior $\Pi_\lambda(\,\cdot\,|\bs x,\bs y)$, i.e.
\[ \max_{Q \in \mathcal Q_\lambda} \ELBO_\lambda(Q) = \min_{Q \in \mathcal Q_\lambda} \KL(Q \| \Pi_\lambda(\,\cdot\, | \bs x,\bs y)). \]
This objective can be maximised both over $Q$ and $\lambda$, but since the objective function changes with $\lambda$, the two steps cannot be merged into one as before in \eqref{def:vb:posterior}. However, if the prior distributions $\Pi_\lambda$, $\lambda\in\Lambda_n$ are all mutually singular, there is a simple relationship between the two ELBO objective functions, i.e.
\begin{equation}\label{e:ELBO.rel}
\ELBO(Q) = \ELBO_\lambda(Q) + \log \pi(\lambda).
\end{equation}
As a consequence the variational posterior $\tilde\Pi_{\lambdavar}(\,\cdot\,|\bs x,\bs y)$ can also be obtained by maximising $\ELBO_{\lambda}$ penalized by the log-hyper-prior density. In case the hyper-prior is the uniform distribution on $\Lambda_n$, the two ELBOs are constant shifts of each other.

We argue, however, that this $\ELBO_\lambda$ function is less natural for selecting the hyperparameter than the previous $\ELBO$ function in \eqref{def:ELBO:func}: note that
\begin{equation}\label{e:ELBO.lambda}
\KL(Q \| \Pi_\lambda(\,\cdot\,|\bs x,\bs y)) = \int \log p_f^{(n)}(\bs x,\bs y) \,d\Pi_\lambda(f) - \ELBO_\lambda(Q).
\end{equation}
Therefore, maximizing the evidence lower bound $ \ELBO_\lambda(Q)$ over the hyperparameter $\lambda$ is not equivalent to minimizing the  KL-divergence due to the dependence of the log-evidence term in \eqref{e:ELBO.lambda} on $\lambda$. Nevertheless, the relation \eqref{e:ELBO.rel} shows that the two ELBO-optimisations over $\lambda$ are equivalent if in addition to the mutual singularity the hyper-prior is a uniform distribution. Any other choice of the hyper-prior $\pi$ gives the maximiser of \eqref{e:ELBO.rel} the interpretation of a regularised version of the maximiser of $\ELBO_\lambda$. See also Section 4 in \cite{zhang2020}, who start from the evidence lower bound given in \eqref{e:ELBO.rel}.

\section{Oracle rate with variational Gaussian processes}\label{c:adaptive.rates}

In this section we derive oracle contraction rates for the variational approximation of Gaussian process mixtures under general conditions. Then, in the next section we provide examples where the variational Bayes approach achieves the minimax adaptive contraction rates for various choices of GP priors, based on these general results. Since the variational approximation aims to mimic the behaviour of the hierarchical posterior, we start by deriving theoretical guarantees for this fully Bayesian approach. Since, to deal with the variational approach, we need tighter control on the tail behaviour of the hyper-posterior $\pi(\,\lambda\,| \bs x, \bs y)$, we do not follow standard techniques (as in e.g. \cite{jonge2010, arbel2013}), but use an argument via empirical Bayes, as developed in \cite{szabo2013, rousseau2017}. 

Since our focus is on GP approximations, for computational and analytical convenience we tailor our conditions to centered GP priors $\Pi_\lambda$ with covariance kernel $k_\lambda(\cdot,\cdot)$ and discrete hyper-priors $\pi$ on $\lambda$, i.e. we study mixtures of the form
\[
\Pi = \sum_{\lambda \in \Lambda_n} \pi(\lambda) \Pi_\lambda.
\]
For each value of the hyperparameter $\lambda$ we associate to the conditional prior $\Pi_\lambda$ a rate $\eps_n(\lambda)$ defined as the solution to
\begin{equation}\label{e:eps.lambda}
	\Pi_\lambda(f \in L^2(\mathcal X,G) : \|f - f_0\| \leq K\eps_n(\lambda))
	= \exp(-n\eps_n^2(\lambda)),
\end{equation}
where $K$ is a large positive constant to be specified later. Under some mild additional assumptions, one can show that  $\eps_n(\lambda)$ is the contraction rate of the posterior $\Pi_\lambda(\,\cdot\,|\bs x,\bs y)$ (this follows from the theory below applied to $\Lambda_n = \{\lambda\}$). Under the conditions below we prove that the hierarchical posterior distribution contracts at the rate
\[ \eps_{n,0} := \delta_n \vee c_n \min_{\lambda\in\Lambda_n} \eps_n(\lambda), \]
where $c_n$ is a sequence tending to infinity arbitrarily slowly and $\delta_n$ is a threshold which ensures $\eps_{n,0}$ is bounded from below. One can view this rate as the oracle one, i.e. the best rate attained by any of the models/priors. In the next section we provide several examples where this oracle rate is minimax adaptive over a scale of Sobolev regularity classes. 

The contraction rate of the posterior distribution is measured in the norm $\|\cdot\|$ of $L^2(\mathcal X,G)$, but in the proof we also use a standard testing argument, built on the empirical norm $\|g\|_n = (n^{-1} \sum_{i=1}^n g(x_i)^2)^{1/2}$. In order to compare these two, we introduce the following regularity condition on the Gaussian process prior. Consider for any $\lambda$ the Karhunen-Loève expansion of the Gaussian process with law $\Pi_\lambda$,
\begin{equation}\label{def:GP:prior}
\sum_{j=1}^{\infty} s_j^\lambda Z_j \varphi_j,\qquad Z_j\overset{\mathrm{iid}}{\sim}\mathcal N(0,1),
\end{equation}
where $(\varphi_j)_{j\in\mathbb{N}}$ denotes the eigenbasis of the covariance kernel that defines $\Pi_\lambda$, and $(s_j^\lambda)_{j\in\mathbb{N}}$ are the square roots of the eigenvalues. We require that the suprema of the eigenfunctions increase at most at a polynomial rate. That is, we assume there exist $\gamma\geq0$ and $C_\varphi>0$ such that
\begin{equation}\label{e:eigBound}
\|\varphi_j \|_\infty \leq C_\varphi j^{\gamma/d} \text{ for all } j \in \bb N.
\end{equation}
This condition is used together with a tail bound for the prior given below. Note that \eqref{e:eigBound} is a mild assumption which is satisfied for example by the standard Fourier basis with $\gamma=0$. For a general kernel, the eigenfunctions may depend on the hyperparameter(s), in which case the results below still hold if the inequality \eqref{e:eigBound} is valid uniformly for all $\lambda$ in $\Lambda_n$. 

Given a function $g = \sum_j \ip{g,\varphi_j} \varphi_j \in L^2(\mathcal X, G)$, we denote its tail (in the spectral domain) by
\[
	g^{>J} := \sum_{j > J} \ip{g,\varphi_j}\varphi_j.
\]
For $J_\gamma = (n/(\log n)^2)^{d/(d+2\gamma)}$ with $\gamma$ as in condition~\eqref{e:eigBound}, suppose that
\begin{equation}\label{e:f0.tail}
\|f_0^{>J_\gamma}\|_\infty \leq \delta_n.
\end{equation}
For the prior we require a similar tail condition:
\begin{equation}\label{e:priortail}
	\max_{\lambda\in\Lambda_n} \Pi_\lambda(\|f^{>J_\gamma}\|_\infty > \delta_n) \leq e^{-n\eps_{n,0}^2}.
\end{equation}
These assumptions are used to compare the empirical norm $\|g\|_n = (n^{-1}\sum_{i=1}^n g(x_i)^2)^{1/2}$ used for testing with the $L^2$ norm $\|g\| = (\int g^2 d G)^{1/2}$ on $L^2(\mathcal X,G)$ in which we measure the contraction rate (recall the definition of $\eps_n(\lambda)$ in \eqref{e:eps.lambda}). Loosely speaking, we compare the two norms for functions projected onto the basis $\varphi_1,\ldots,\varphi_{J_\gamma}$ and handle the tail behaviour with the above upper bounds. We verify these assumptions in several examples of regularity classes and priors under some conditions depending on the parameter $\gamma$.

Furthermore, we assume that there exists a function $c:(0,1)\mapsto \mathbb{R}$ such that for any $\lambda \in \bigcup_n \Lambda_n$, $\xi\in(0,1)$ and all sufficiently small $\eps>0$,
\begin{equation}\label{e:prior.reg.sc}
- \log \Pi_\lambda(\|f\| \leq \xi \eps) \leq - c(\xi) \log \Pi_\lambda(\|f\| \leq \eps).
\end{equation} 
This condition is used to refine the metric entropy bound implied by \eqref{e:eps.lambda} in order to construct hypothesis tests whose error is sufficiently small.

Lastly, we impose a condition on the hyper-prior $\pi$. We define the subset of `good' hyperparameters as
\begin{equation}\label{e:Lambda0}
	\Lambda_{n,0} := \{ \lambda \in \Lambda_n : \eps_n(\lambda) \leq \eps_{n,0} \},
\end{equation}
and assume that the hyper-prior puts sufficient amount of mass on each element of this set, i.e.
\begin{equation}\label{e:hyper.mass}
	-\log \Big(\min_{\lambda \in \Lambda_{n,0}} \pi(\lambda)\Big) = o(n\delta_n^2).
\end{equation}
The above condition on the hyper-prior is natural, and moreover very mild, as it only requires that at least an exponentially small mass is put on the elements $\lambda \in \Lambda_{n,0}$. It is satisfied for a wide range of distributions, including the uniform distribution on a set $\Lambda_n$ that does not grow too fast with $n$. 
Furthermore, we denote by $\bb E_{f_0}$ the expectation of the data $\bs x,\bs y$ in the model \eqref{e:np.reg}. We are now in a position to present the main result on contraction for the mixture posterior.

\begin{theorem}\label{t:hbp.contr}
Suppose that $\min_{\lambda\in\Lambda_n} n\eps_n^2(\lambda) \to \infty$ and $\log |\Lambda_n| = o(n\delta_n^2)$. Then under conditions~\eqref{e:eigBound}, \eqref{e:f0.tail}, \eqref{e:priortail}, \eqref{e:prior.reg.sc} and \eqref{e:hyper.mass}, there exist a constant $M>0$ and an event $\mathcal A_n$ with $\bb P_{f_0}^{(n)}(\mathcal{A}_n^c)=o(1)$ such that
\[
	\bb E_{f_0} \Pi (\|f-f_0\| > M \eps_{n,0} | \bs x, \bs y) \bs 1_{\mathcal{A}_n}
	\lesssim e^{-n\eps_{n,0}^2/3}.
\]
\end{theorem}
Note that the above theorem implies that the hierarchical posterior contracts at the rate $\min_{\lambda\in\Lambda_n} \eps_n(\lambda) \vee \delta_n$ around the true parameter $f_0$, as the sequence $c_n \to \infty$ in the definition of $\eps_{n,0}$ is arbitrary slow. In fact, the result is slightly stronger as it derives an exponentially fast convergence on a large event $\mathcal{A}_n$. 

The proof of the theorem follows the lines of \cite{rousseau2017} adapted to our specific setting. We start by showing that the maximum marginal likelihood estimator (MMLE)
\begin{align}\label{def:MMLE}
\hat\lambda_n=\arg\max_{\lambda\in\Lambda_n} \int p_f^{(n)}(\bs x,\bs y)d\Pi_\lambda(f)
\end{align}
belongs to the set of optimal hyperparameters $\Lambda_{n,0}$ with probability tending to one. This implies as an ancillary result that the MMLE empirical Bayes posterior
\begin{align}
\Pi_{\hat\lambda_n}(\,\cdot\,|\bs x,\bs y)=\Pi_{\lambda}(\,\cdot\,|\bs x,\bs y)\Big|_{\lambda=\hat\lambda_n}\label{def:eb}
\end{align}
achieves the optimal oracle rate $\eps_{n,0}$; see Theorem \ref{t:ebp.contr} below. Furthermore, under the assumption that the set of optimal hyperparameters $\Lambda_{n,0}$ receives a sufficiently large prior mass \eqref{e:hyper.mass} the hierarchical posterior also attains the same oracle contraction rate $\eps_{n,0}$. The details are given in Section~\ref{proof:t:hbp.contr}.

Next we investigate the variational approximation of the hierarchical posterior given in \eqref{def:vb:posterior}.

\begin{theorem}\label{t:vbp.contr}
In addition to the conditions of Theorem~\ref{t:hbp.contr}, assume that there exists $\lambda_n \in\Lambda_{n,0}$ and $Q_{\lambda_n} \in \mathcal Q_{\lambda_n}$ such that $\eps_n(\lambda_n) = o(\eps_{n,0})$ and
\begin{equation}\label{e:vbcCondition}
	\bb E_{f_0} \KL(Q_{\lambda_n} \| \Pi_{\lambda_n}(\,\cdot\,|\bs x,\bs y)) = o(n\eps_{n,0}^2).
\end{equation}
Then there exists $M>0$ such that
\[
\bb E_{f_0} \tilde \Pi_{\hat\lambda_n^{\mathrm{var}}}(f : \|f-f_0\| \geq M \eps_{n,0} |\bs x,\bs y) \to 0.
\]
\end{theorem}
The theorem states that under Condition \eqref{e:vbcCondition}, the variational posterior with the empirical hyperparameter maximizing the ELBO function in \eqref{e:lambda.hat.var} has the same rate of contraction as the hierarchical posterior. As seen in the examples below, Condition \eqref{e:vbcCondition} is mild. It basically requires that there is at least one hyperparameter amongst the ''good ones'' in $\Lambda_{n,0}$ such that the variational class is sufficiently close in Kullback-Leibler divergence to the corresponding posterior distribution. The proof is deferred to Section~\ref{proof:t:vbp.contr}.

In the literature various, related conditions to \eqref{e:vbcCondition} were considered. Recalling \eqref{e:ELBO.lambda}, note that
\begin{align*}
& \bb E_{f_0} \KL(Q \| \Pi_\lambda(\,\cdot\, | \bs x,\bs y)) \\
& = \bb E_{f_0} \int \log \frac{p_f^{(n)}}{p_{f_0}^{(n)}}(\bs x,\bs y) \,d\Pi_\lambda(f) - \bb E_{f_0} \Big( \int \log \frac{p_f^{(n)}}{p_{f_0}^{(n)}}(\bs x,\bs y) \,dQ(f) - \KL(Q \| \Pi_\lambda) \Big) \\
& \leq \KL(Q \| \Pi_\lambda) + \int \KL(\bb P_{f_0}^{(n)} \| \bb P_f^{(n)}) \,dQ(f).
\end{align*}
Assuming that this upper bound is of the order $O(n\eps_{n,0}^2)$ is similar to e.g. condition (C4) in \cite{zhang2020} for variational posteriors and the assumptions in Theorem 2.4 of \cite{alquier2020} for tempered variational posteriors.

\section{Examples}\label{s:ex}

In this section we apply the main result on variational GP contraction to two specific variational Bayes methods. 

The data-generating true function $f_0$ is assumed to belong to a $\beta$-Sobolev type ball
\[ \mathcal S^\beta(L) := \{g \in L^2(\mathcal X,G) : {\textstyle \sum{}} j^{2\beta/d} \ip{g,\varphi_j}^2 \leq L^2 \}, \]
i.e., its Sobolev norm $\|f_0\|_\beta := (\sum j^{2\beta/d}\ip{f_0,\varphi_j}^2)^{1/2}$ is bounded by $L$. We call the variational posterior \textit{adaptive} if for $\beta\in B$, for some subset $B\subset\mathbb{R}_+$,
\[
\sup_{f_0 \in \mathcal S^\beta(L)} \bb E_{f_0} \tilde\Pi_{\hat\lambda_n^{\mathrm{var}}}(f : \|f-f_0\| \leq M_n n^{-\beta/(d+2\beta)} | \bs x,\bs y) \to 0
\]
as $n\to\infty$, where $M_n$ is an arbitrary sequence tending to infinity. We show through several examples below that various choices of priors and variational classes can result in adaptive contraction rates. First we consider the inducing variable methods introduced in \cite{titsias2009} for polynomially and exponentially decaying eigenvalues of the prior covariance kernel. Then we also study a mean-field approximation by taking a mixture of truncated GPs as the hierarchical prior and truncated, mean-field GPs as the variational class.

\subsection{Inducing variables approximations}
First we consider the inducing variable variational approximations for GP regression introduced by \cite{titsias2009}. The idea is to compress the information possessed by the posterior into $m$ variables $\bs u= (u_1,\ldots,u_m)$, which in turn will reduce the computational complexity of the method. The inducing variables are taken to be linear functionals of the parameter $f$, hence $\bs u$ possesses an $m$-dimensional Gaussian distribution and $f|\bs u$ is a Gaussian process. Then endowing $\bs u$ with a Gaussian distribution with mean $\mu$ and variance $\Sigma$ and integrating out the conditional distribution $f|\bs u$ results in a class of Gaussian processes indexed by $\mu\in\mathbb{R}^m$ and $\Sigma\in\mathbb{R}^{m\times m}$. The mean and covariance functions of the inducing GPs are of the form 
\begin{align*}
	x &\mapsto K_{x\bs{u}}K_{\bs{uu}}^{-1}{\mu},\\
	(x,x') & \mapsto k_\lambda(x,x') - 
	K_{x\bs{u}}K_{\bs{uu}}^{-1}(K_{\bs{uu}}-\Sigma)K_{\bs{uu}}^{-1}K_{\bs{u}x'},
\end{align*}
where $k_\lambda(\cdot,\cdot)$ denotes the covariance kernel of the prior, $K_{x\bs{u}}\in \mathbb{R}^{m}$ and $K_{\bs{uu}} \in \bb R^{m\times m}$ the covariance matrices (implicitly also dependent on $\lambda$) corresponding to $f(x)$ and the vector $\bs u\in\mathbb{R}^m$. We use this class of GPs as our variational class, i.e. we take
\begin{equation}\label{e:varClass.ind}
\mathcal Q_\lambda = \Big\{ \int \Pi_\lambda(\,\cdot\,| \bs u) \, d \mathcal N_{\mu,\Sigma}(\bs u): \mu\in\mathbb{R}^m, \Sigma\in\mathbb{R}^{m\times m} \Big\},
\end{equation}
where $\mathcal N_{\mu,\Sigma}$ denotes the Gaussian distribution with mean $\mu$ and covariance $\Sigma$.

Various choices of inducing variables can be considered. Here we focus on the population and empirical spectral features inducing variables. In the  \textit{population spectral features} method we take $u_j = \ip{f,\varphi_j}$, $j=1,\ldots,m$ as inducing variables, where $\varphi_j$ denotes the $j$th eigenfunction of the prior covariance kernel $k_\lambda(\cdot,\cdot)$. The computational complexity of this approach is $O(m^2n)$, substantially reducing the computational time. However, this approach requires the explicit knowledge of the eigenfunctions of the prior GPs, which are often not available analytically. The empirical version of this method is called the \textit{sample spectral features} variational method, where the inducing variables are defined as $u_j = \bs v_j^{\text T} f(\bs x)$, where $\bs v_j$ is the $j$-th principal vector of the prior covariance matrix with entries $k_\lambda(x_i,x_j)_{i,j=1,\ldots,n}$. One can see that the covariance kernel is replaced with the sample covariance matrix and the eigenfunctions with the eigenvectors in this approach. This requires computation of the first $m$ principal components $\bs v_j$ and the corresponding eigenvalues, which makes the procedure slower. But it is also more widely applicable than the population features, because it does not require an explicit expression for the eigenfunctions $\varphi_j$. In practice the $\bs v_j$ are computed via Lanczos iteration or conjugate gradient descent, see for instance \cite{gardner2018, wenger2022, stankewitz2024}.

For fixed $\lambda$, the variational posterior can be determined analytically in terms of the data and covariances $K_{\bs u \bs u}$, $K_{\bs u x_i}$, which can be computed exactly or numerically depending on the choice of inducing variables. For non-adaptive variational approximations the theoretical properties have been studied in \cite{burt2020, vbgp}.

Below we consider two of the arguably most standard eigenvalue structures, i.e. polynomially and exponentially decaying eigenvalues $(s_j^\lambda)_{j\in\mathbb{N}}$ for the GP priors.

\subsubsection{Polynomially decaying eigenvalues -- tuning the exponent}
First we investigate priors with polynomially decaying eigenvalues of the form
\begin{equation} \label{def:prior:poly}
	f \sim \sum_{j=1}^\infty s_j^\alpha Z_j \varphi_j, \qquad s_j^\alpha \asymp j^{-1/2-\alpha/d},
\end{equation}
where $Z_1,Z_2,\ldots$ are i.i.d. standard normal, $(\varphi_j)_{j \in \bb N}$ is an orthonormal basis satisfying condition \eqref{e:eigBound} for some $\gamma\geq 0$, $d$ is the dimension of the covariate vectors $x_i$, and $\lambda = \alpha > 0$ is the regularity of the prior. Several priors possess such eigenstructure, including the Mat\'ern kernels and the Riemann-Liouville processes (including integrated Brownian motions). Here the form of the prior  closely matches the structure of the smoothness class, and for $\alpha=\beta$ the original posterior $\Pi_\alpha(\,\cdot\,|\bs x,\bs y)$ contracts at the minimax rate $n^{-\beta/(d+2\beta)}$ around $f_0 \in \mathcal S^\beta(L)$. In \cite{vbgp} it was shown that the variational posterior $\tilde \Pi_{\alpha=\beta}(\,\cdot\, | \bs x,\bs y)$ using either population or empirical spectral features inducing variables, contracts at this rate (albeit with respect to the Hellinger distance) provided that the number of inducing variables exceeds $m_{n,\alpha} \geq n^{d/(d+2\alpha)}$. Here we extend these results to data driven tuning of the hyperparameter and derive rates
 with respect to the stronger $L^2$-norm instead of Hellinger distance. We consider the variational class in \eqref{e:varClass.ind} with $m_{n,\alpha}$ the smallest integer larger than $n^{d/(d+2\alpha)}$. 

We optimize the hyperparameter $\alpha$ over the set
\[ \Lambda_n := \{ \beta^- + k/\log n : k=0,1,\ldots \} \cap [\beta^-,\beta^+], \]
for fixed $\beta^+ > \beta^- > 0$. We assume that $f_0 \in \mathcal S^\beta(L)$ for a $\beta\in  [\beta^-,\beta^+]$. We consider a uniform discrete prior $\pi$ on $\Lambda_n$.

\begin{corollary}\label{cor:poly}
Let $\beta^- > d+2\gamma$, where $\gamma>0$ is as in condition \eqref{e:eigBound}, and assume that $f_0\in \mathcal S^\beta(L)$ for some $\beta\in  [\beta^-,\beta^+]$. Then the population and empirical spectral features inducing variable variational posteriors corresponding to the prior \eqref{def:prior:poly}  achieve the minimax contraction rate $n^{-\beta/(d+2\beta)}$, i.e.
\[
	\bb E_{f_0} \tilde\Pi_{\hat\alpha_n^{\mathrm{var}}}(\|f-f_0\| \geq M_n n^{-\beta/(d+2\beta)}
	| \bs x,\bs y) \to 0
\]
for an arbitrary sequence $M_n$ tending to infinity.
\end{corollary}

The proof of the corollary is deferred to Section \ref{sec:cor:poly} in the Appendix.

As an example of these theoretical findings we consider a simulated data set of $n=10,000$ observations $(x_i,y_i)$ generated as 
\begin{equation}\label{e:sim.poly.data}
x_i \sim_{\text{i.i.d.}} \mathrm{unif}(0,2\pi) \text{ and } y_i | x_i \sim_{\text{ind.}} \mathcal N(f_0(x_i), 0.01),
\end{equation}
where
\begin{equation}\label{e:sim.poly.f0}
f_0(x) = \sum_{j=0}^\infty (3j+1)^{-1.1} \varphi_{3j+1}(x-\pi)
\end{equation}
and $\{\varphi_j\}_{j\in\mathbb{N}}$ form the real Fourier basis on $L^2(0,2\pi)$ rescaled to be orthonormal with respect to the uniform distribution. The function $f_0$ almost has Sobolev smoothness $0.6$ in the sense that $f_0 \in \mathcal S^\beta(L)$ for all $\beta < 0.6$. We take the prior defined in \eqref{def:prior:poly}, and using the ELBO with population spectral features we estimate $(\sigma^2,\lambda)$ by $(0.014, 1.006)$ and correspondingly take $m = 21$. In Figure~\ref{f:m21} we show $f_0$ and the mean and $95\%$ pointwise credible regions of the variational posterior with variationally tuned $\lambda=\alpha$. The credible regions nicely cover the true $f_0$ everywhere except for some of the points where it changes direction. 

\begin{figure}[h]\centering
\includegraphics{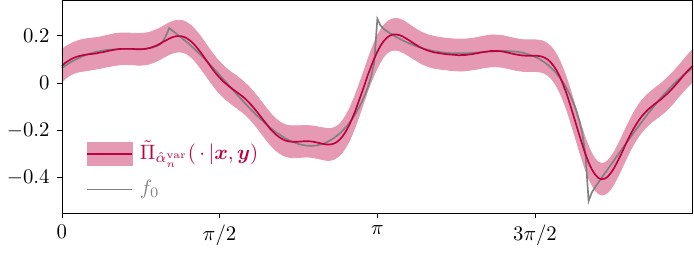}
\caption{Variational posterior with $m=21$ population features based on polynomially decaying series prior \eqref{def:prior:poly}. Data simulated as in \eqref{e:sim.poly.data} with $f_0$ given in \eqref{e:sim.poly.f0}. The shaded region indicates $95\%$ pointwise credible regions of the variational posterior.}
\label{f:m21}
\end{figure}

In Figure~\ref{f:m42} we repeat the procedure (with new data) but set $m = \lceil 2n^{1/(1+2\alpha)} \rceil$ in the variational class. Now $(\sigma^2,\lambda)$ has similar estimates $(0.013,1.005)$ implying that $m=42$. Generally we observe similar behaviour as before, but the credible regions are narrower. Based on this we conclude that the constant multiplier in the number of inducing variables can play a non-negligible role. In practice the exact choice of $m$ is delicate. On the one hand one should choose it as large as the computational resources allow. Increasing $m$ beyond a threshold effectively doesn't help anymore, but this threshold depends on certain hidden properties of the underlying signal. Our theoretical results provide an initial guide for choosing $m$. Then, in practice, if the computational resources allow, one can increase $m$ gradually until no significant changes can be observed. This approach combines the asymptotic results with the finite sample size behaviour, making the procedure more robust.

\begin{figure}[h]\centering
\includegraphics{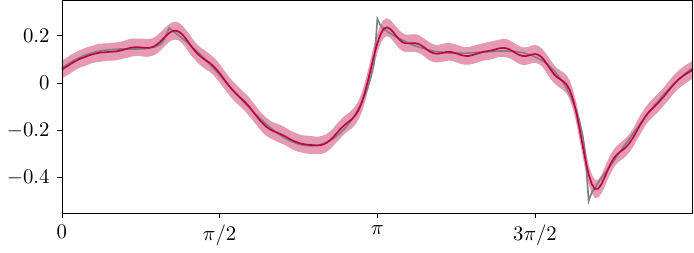}
\caption{Setting exactly as in Figure~\ref{f:m21}, but with $m=42$ features.}
\label{f:m42}
\end{figure}

\subsubsection{Exponentially decaying eigenvalues -- adaptive rescaling}

Now consider a series prior with eigenvalues that decay exponentially
\begin{equation}\label{e:expoPrior}
f \sim \sum_{j=1}^\infty s_j^\tau Z_j \varphi_j, \qquad \quad s_j^\tau \asymp (\tau^d \exp(-\tau j^{1/d}))^{1/2},
\end{equation}
where $Z_1,Z_2,\ldots$ are i.i.d. standard normal random variables and $(\varphi_j)_{j \in \bb N}$ is an orthonormal basis satisfying \eqref{e:eigBound}. A prominent example of such eigenvalue structure is the rescaled squared exponential covariance kernel. In this context, we estimate the scaling hyperparameter $\lambda = \tau$. If the true Sobolev smoothness of $f_0$ is $\beta$, then (up to logarithmic factors) the optimal rescaling $\tau$ is of the order $n^{-1/(d+2\beta)}$ and one should take at least $n^{d/(d+2\beta)}$ inducing variables in the approximation (see \cite{vbgp}). We study here the case where $\beta$ is not known, and estimate $\tau$ empirically by maximizing the ELBO over the set
\[ \Lambda_n = \{ e^{-j} : j=1,\ldots,\lfloor \tfrac{1}{d+2\beta^-} \log n \rfloor \}, \]
as in \eqref{e:lambda.hat.var}, where $\beta^-$ is a lower bound on the true Sobolev smoothness of $f_0$, which is assumed to satisfy $\beta^- > d+2\gamma$. As before, the inducing variables class \eqref{e:varClass.ind} is studied for both the empirical and population spectral features, with $m_{n,\tau} = \lceil \tau^{-d} (\log n)^{d+1} \rceil$ features. 

\begin{corollary}\label{cor:exp}
Suppose the prior basis $(\varphi_j)_{j \in \bb N}$ in \eqref{e:expoPrior} satisfies condition \eqref{e:eigBound} for some $\gamma>0$. If $f_0\in \mathcal S^\beta(L)$ for some $\beta \geq \beta^-$, then the population and empirical spectral features inducing variables variational posteriors corresponding to the prior \eqref{e:expoPrior} achieve the (near) minimax contraction rate $n^{-\beta/(d+2\beta)}$, i.e.
\[
	\bb E_{f_0} \tilde \Pi_{\hat\tau_n^{\mathrm{var}}}(\|f-f_0\| \geq M_n (n/\log n)^{-\beta/(d+2\beta)}
	| \bs x,\bs y) \to 0
\]
for an arbitrary sequence $M_n$ tending to infinity.
\end{corollary}
The proof of the corollary is deferred to Section \ref{sec:cor:exp} in the Appendix.

\newpage 

\begin{figure}[h]\centering
\includegraphics{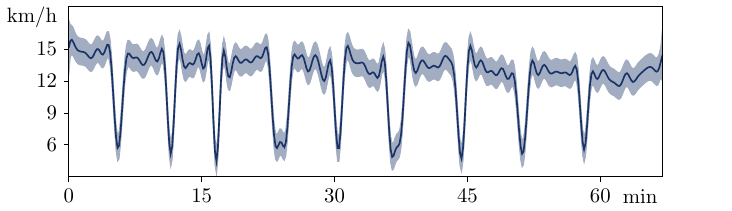}
\caption{Mean and 95\% credible region of the variational posterior for runner's speed based on the squared exponential prior \eqref{e:kernel.sq}. A variational approximation was used with $m=150$ sample spectral features, where hyperparameters and model variance were tuned using the evidence lower bound. Optimisation time: 286 sec.}
\label{f:running}
\end{figure}

~

\begin{figure}[h]\centering
\includegraphics{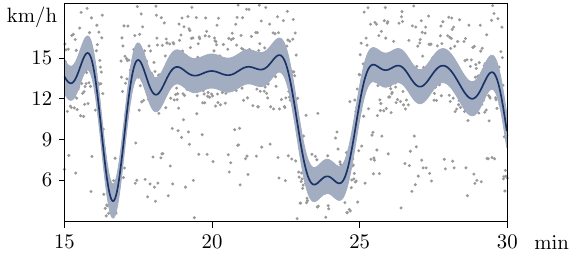}
\caption{Zoom of variational procedure in Figure~\ref{f:running} with plot of data points.}
\label{f:runningzoom}
\end{figure}

~

\begin{figure}[h]\centering
\includegraphics{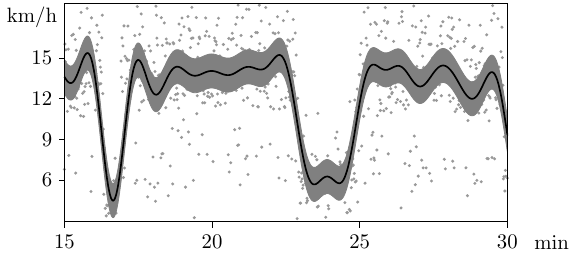}
\caption{Empirical Bayes posterior mean and point-wise $95\%$ credible region for the running data with squared exponential prior \eqref{e:kernel.sq}. Optimisation time: 817 sec.}
\label{f:empbayes}
\end{figure}

\newpage

We illustrate the results by applying the procedure to a set of $n=4022$ velocity measurements during a run\footnote{Code and data are available at \url{https://github.com/dennisnieman/adaptiveVB}}. A GPS tracker has observed the runner's coordinates every second during their interval training and the velocity in $\mathrm{km}/\mathrm{h}$ is estimated from the coordinate shifts. The measurements are smoothed using the squared exponential kernel
\begin{equation}\label{e:kernel.sq}
k(x,y) = \nu \cdot \exp(-(x-y)^2/\tau^2).
\end{equation}
Here $\nu$ and $\tau$ are respectively a vertical and a sample path scaling parameter. Since the eigenfunctions in the series expansion of this GP are not known explicitly, we apply the variational approximation with sample spectral features. Both the hyperparameters and the model variance $\sigma^2$ are estimated variationally by maximising the evidence lower bound. Although our theory is developed only for variational estimation of $\tau$, in practice we also obtain reasonable estimates for the other unknowns, which are comparable to empirical Bayes estimates.

In comparison with the preceding section, here it is not quite clear how $m$ should depend on the tuning parameters, since our theory only accounts for the scaling parameter $\tau$. Therefore we cautiously take large $m=150$. The sample spectral features were computed using an off-the-shelf eigendecomposition method, which computes the entire decomposition. This is somewhat inefficient, especially since every evaluation of the ELBO requires computation of a new decomposition. Yet already the procedure is almost three times faster than the empirical Bayes procedure. The estimates of $(\sigma, \nu, \tau)$ are $(3.88, 99.58, 59.24)$. From the data it becomes clear that the changes in speed are quite abrupt, which is reflected quite well in the variational posterior (see Figure~\ref{f:runningzoom}). In general, plain stationary GPs are not completely suited to detect jumps and varying local behaviours, and the fit could be improved with a model taking into account change-points.

We compare our variational approximation with the empirical Bayes approach in Figure~\ref{f:empbayes}. We estimate the hyperparameters $(\sigma, \nu, \tau)$  by maximizing the marginal likelihood function  $(\sigma,\nu,\tau) \mapsto \int \log p_f^{(n)}(\bs x,\bs y) \,d\Pi_\lambda (f)$, resulting in estimates $(3.88, 99.59, 59.17)$. By eye, the variational approximation can hardly be distinguished from the empirical Bayes posterior.

\subsection{Truncated series priors}

As third example, consider a prior of the form
\begin{equation}\label{e:dimprior}
f \sim D^{-1/2} \sum_{j=1}^D Z_j \varphi_j, \qquad Z_j \sim_{\text{i.i.d.}} \mathcal N(0,1),
\end{equation}
for a uniformly bounded basis $\varphi_1,\varphi_2,\ldots$ (that is, satisfying condition \eqref{e:eigBound} for $\gamma = 0$). We study a variational procedure for the tuning parameter $\lambda = D \in \bb N$.

Consider a uniform hyper-prior on the set
\[
	\Lambda_n = \{ i \in \bb N : i \leq \sqrt{n} \}.
\]
In a practical application, the optimisation is faster if the grid is coarser. For example, one might consider the subset of those integers in $\Lambda_n$ of the form $2^i$ where $i \in \bb N$. This may worsen the quality of the approximation, even though asymptotic results are unaffected by such a change (as long as a prior is chosen for $D$ that satisfies condition \eqref{e:hyper.mass}). 

For fixed $D$, the posterior distribution is a Gaussian process with respective mean and covariance function
\begin{align*}
x & \mapsto \varphi_{1:D}(x)^\T(\sigma^2 D I + \Phi^\T\Phi)^{-1}\Phi^\T \bs y, \\
(x,y) & \mapsto \varphi_{1:D}(x)^\T(D I + \sigma^{-2} \Phi^\T\Phi)^{-1} \varphi_{1:D}(y),
\end{align*}
where we abbreviate $\varphi_{1:D}(x) = (\varphi_1(x),\ldots,\varphi_D(x))$, and denote by $\Phi$ the $n \times D$ matrix whose $j$-th row is $\varphi_{1:D}(x_i)^\T$.

Note that the matrix $\Phi^\T\Phi$ has entries $\sum_{i=1}^n\phi_j(x_i)\phi_k(x_i)$ which can be approximated by $n\delta_{jk}$ using the Law of Large Numbers. Thus the posterior distribution is a $D$-dimensional process, and a posteriori the coefficients $\ip{f,\varphi_j}$ are nearly diagonal. This motivates the use of the variational class
\[ \mathcal Q_D = \{ \mathcal N(\sigma^{-2}\Sigma\Phi^\T \bs y, \Sigma) : \Sigma \in \bb R^{D\times D} \text{ positive definite and diagonal} \}, \]
here viewed as a distribution on the coefficients in the basis $\varphi_1,\ldots,\varphi_D$. The variational class respects the relationship between the mean and variance of the original posterior. It can also be used to approximate the posterior from a more general Gaussian series prior with coefficients $\ip{f,\varphi_j} = s_j^D Z_j$ instead of $D^{-1/2}Z_j$, where the $s_j^D$ have a decaying structure. For simplicity we restrict ourselves the truncated GP of the form \eqref{e:dimprior}. 

Denoting $\mu = \sigma^{-2}\Sigma\Phi^\T \bs y$, the variational distribution $\tilde\Pi_D(\,\cdot\,|\bs x,\bs y)$ is the maximiser of
\begin{align*}
	\ELBO(Q)
	& = \int \log p_f^{(n)}(\bs x,\bs y) \,dQ(f) - \KL(Q \| \Pi_D) + \log \pi(D) \\
	& = - \frac n2 \log(2\pi\sigma^2) - \frac{1}{2\sigma^2}\Big( \tr \Phi\Sigma\Phi^T + (\Phi\mu - \bs y)^\T(\Phi\mu - \bs y)\Big) \\
	& \phantom{={}} - \frac 12 \Big(- \log |D\Sigma| - D + \tr(D\Sigma) + D\mu^T\mu\Big) + \log \pi(D)
\end{align*}
over $Q \in \mathcal Q_D$. Next the variational estimator $\hat D_n^{\mathrm{var}}$ is determined as the maximiser of the function $D\mapsto \ELBO(\tilde \Pi_D(\,\cdot\,|\bs x,\bs y))$.

\begin{corollary}\label{cor:dim}
Suppose the prior basis $(\varphi_j)_{j \in \bb N}$ in \eqref{e:dimprior} satisfies condition \eqref{e:eigBound} for $\gamma= 0$. If $f_0\in \mathcal S^\beta(L)$ for some $\beta > d/2$, then the variational posterior $\tilde\Pi_{\hat D_n^{\mathrm{var}}}(\,\cdot\,|\bs x,\bs y)$ achieves the (near) minimax contraction rate $n^{-\beta/(d+2\beta)}$, i.e.
\[
	\bb E_{f_0} \tilde \Pi_{\hat D_n^{\mathrm{var}}}(\|f-f_0\| \geq M_n (n/\log n)^{-\beta/(d+2\beta)}
	| \bs x,\bs y) \to 0
\]
for an arbitrary sequence $M_n$ tending to infinity.
\end{corollary}
The proof is given in Appendix~\ref{sec:cor:dim}.

\acks{Co-funded by the European Union (ERC, BigBayesUQ, project number: 101041064). Views and opinions expressed are however those of the author(s) only and do not necessarily reflect those of the European Union or the European Research Council. Neither the European Union nor the granting authority can be held responsible for them.

We cordially thank Harry van Zanten for the fruitful discussions in the early stages of writing this paper.}

\appendix

\section{Proof of general contraction rate theorems}\label{sec:main:proofs}
In this section we collect the proofs for the contraction rates of the hierarchical, empirical and variational Bayes posteriors. But as a first step we study the asymptotic behaviour of the  maximum marginal likelihood estimator (MMLE). 

\subsection{Asymptotic behavior of the MMLE}
We show that with $\bb P_{f_0}^{(n)}$-probability tending to one the maximum marginal likelihood estimator $\hat\lambda_n$ defined in \eqref{def:MMLE} belongs to the set of good hyperparameters $\Lambda_{n,0}$ defined in \eqref{e:Lambda0}, where the corresponding posteriors $\Pi_\lambda(\,\cdot\,|\bs x,\bs y)$ nearly attain the oracle rate $\eps_{n,0}$. The lemma below is an adaptation of Theorem 2.1 of \cite{rousseau2017} to the random design regression model. A major difference is that we consider general basis functions satisfying certain boundedness and tail assumptions, see below. Since the proof consists several new technical arguments compared to \cite{rousseau2017}, we provide the whole proof for completeness and easier readability. 

\begin{lemma}\label{l:mmleSet}
Suppose that $\min_{\lambda\in\Lambda_n} n\eps_n^2(\lambda) \to \infty$ and $\log |\Lambda_n| = o(n\delta_n^2)$. Then under conditions~\eqref{e:eigBound}, \eqref{e:f0.tail}, \eqref{e:priortail}, \eqref{e:prior.reg.sc} we have
\begin{equation}\label{e:inLambda0}
\lim_{n\to\infty}\bb P_{f_0}^{(n)}(\hat\lambda_n \in \Lambda_{n,0}) = 1.
\end{equation}
\end{lemma}

\begin{proof}
We bound the evidence $\int p_f^{(n)}/p_{f_0}^{(n)}(\bs x,\bs y) \,d\Pi_\lambda(f)$ from below using Lemma 10 in \cite{ghosal2007}. In their notation with $k=2$ and $\eps=(K/\sigma)\eps_n(\lambda)$, for $\eps_n(\lambda)$ as in \eqref{e:eps.lambda}, it follows that
\[
	B_n(f_0,\eps;2)
	= \{f:\, K(p_{f_0}^{(n)},p_f^{(n)})\leq n\eps^2,\, V_2(p_{f_0}^{(n)},p_f^{(n)})\leq n\eps^2  \}
= \{ f : \|f-f_0\| \leq K\eps_n(\lambda) \},
\]
with $V_2(p,q)=\int p|\log(p/q)|^2d\mu$, where the last equation follows from Lemma 2.7 in \cite{ghosal2017}. Consequently, we obtain for any $\lambda \in \Lambda_n$
\begin{align}
\nonumber
& \bb P_{f_0}^{(n)}\Big( \int \frac{p_f^{(n)}}{p_{f_0}^{(n)}}(\bs x,\bs y) \,d\Pi_\lambda(f) \leq e^{-(1+2K^2/\sigma^2)n\eps_n^2(\lambda)} \Big) \\
\nonumber
& \leq\bb P_{f_0}^{(n)}\Big( \int_{B_n(f_0,\eps;2)} \frac{p_f^{(n)}}{p_{f_0}^{(n)}}(\bs x,\bs y) \,d\Pi_\lambda(f) \leq \Pi_\lambda(B_n(f_0,\eps;2)) e^{-2n\eps^2} \Big) \\
\label{e:lambdaELB}
& \leq \frac{\sigma^2}{K^2n\eps_n^2(\lambda)}.
\end{align}
This holds in particular for $\lambda_n$ that minimises $\eps_n(\lambda)$, that is,
\begin{equation}\label{e:evidenceLB0}
	\bb P_{f_0}^{(n)}\Big(\int \frac{p_f^{(n)}}{p_{f_0}^{(n)}}(\bs x,\bs y) \,d\Pi_{\lambda_n}(f) > \exp(-C_1 n\eps_n^2(\lambda_n))\Big) \to 1,
\end{equation}
where for convenience we define $C_1 := 1 + 2K^2/\sigma^2$. Since $\hat\lambda_n$ is defined as the maximiser of the integral, it follows that
\begin{equation}\label{e:evidenceLB}
\bb P_{f_0}^{(n)}\Big(\int \frac{p_f^{(n)}}{p_{f_0}^{(n)}}(\bs x,\bs y) \,d\Pi_{\hat\lambda_n}(f) > \exp(-n\eps_{n,0}^2/2)\Big) \to 1.
\end{equation}

The proof of \eqref{e:inLambda0} is completed by showing that
\begin{equation}\label{e:evidenceUB}
\bb P_{f_0}^{(n)} \Big( \max_{\lambda \in \Lambda_n\setminus\Lambda_{n,0}} \int \frac{p_f^{(n)}}{p_{f_0}^{(n)}}(\bs x,\bs y) \,d\Pi_\lambda(f) > \exp(-n\eps_{n,0}^2/2)\Big) \to 0.
\end{equation}
Indeed, together with the preceding display this implies that with probability tending to one the evidence is maximised at some $\lambda\in\Lambda_{n,0}$. The proof of \eqref{e:evidenceUB} is done via a prior mass and testing argument. 

\paragraph{Comparison of norms} For $\gamma$ in condition~\eqref{e:eigBound} let $J_\gamma = J_\gamma(n) = (n/\log^2 n)^{d/(d+2\gamma)}$ and $\Phi$ denote the $n \times J_\gamma$ matrix with entries $\Phi_{i,j} = \varphi_j(x_i)$. Then by \cite{rudelson1999},
\[ \bb E_{f_0} \| n^{-1} \Phi'\Phi - \bb I_{J_\gamma} \|_{\mathrm{op}} \lesssim \Big(\frac{\log {J_\gamma}}{n} \sum_{j=1}^{J_\gamma} j^{2\gamma/d}\Big)^{1/2} \lesssim \Big(\frac{J_\gamma^{1+2\gamma/d}\log {J_\gamma}}{n}\Big)^{1/2} \to 0, \]
where $\bb I_{J_\gamma}$ is the ${J_\gamma}$-dimensional identity matrix, and $\|\cdot\|_{\mathrm{op}}$ is the Euclidean operator norm. So the event
\begin{equation}\label{e:Bn}
\mathcal B_n := \{\bs x : \|n^{-1}\Phi'\Phi - \bb I_{J_\gamma}\|_{\mathrm{op}} < 1/2\}
\end{equation}
has $\bb P_{f_0}^{(n)}$-probability going to 1. If $f = \sum_{j=1}^{J_\gamma} \ip{f,\varphi_j}\varphi_j$ and $\bs f = (\ip{f,\varphi_1},\ldots,\ip{f,\varphi_{J_\gamma}})$, then
\[ |\|f\|_n^2 - \|f\|^2| = |\bs f' (n^{-1} \Phi'\Phi - \bb I_{J_\gamma}) \bs f| \leq \|n^{-1}\Phi'\Phi - \bb I_{J_\gamma}\|_{\mathrm{op}} \|f\|^2. \]
More generally, if $f = \sum_{j=1}^\infty \ip{f,\varphi_j}\varphi_j \in L^2(\mathcal X, G)$, then the truncated series $f^{\leq {J_\gamma}} = \sum_{j\leq J_\gamma} \ip{f,\varphi_j}\varphi_j$ satisfies on the event $\mathcal B_n$ 
\begin{equation}\label{e:comparison}
\frac 12 \|f^{\leq {J_\gamma}}\| \leq \|f^{\leq {J_\gamma}}\|_n \leq 2 \|f^{\leq {J_\gamma}}\|.
\end{equation}

\paragraph{Remaining prior mass and entropy} 
For arbitrary $\xi>0$ we construct a set $\mathcal F_n(\lambda) \subseteq L^2(\mathcal X, G)$ such that
\begin{equation}\label{e:lambdaPM}
	\Pi_\lambda(\mathcal F_n(\lambda)^c) \leq \exp(-n\eps_n^2(\lambda))
\end{equation} 
and
\begin{equation}\label{e:lambdaEB}
	\log N(3\xi K\eps_n(\lambda), \mathcal F_n(\lambda), \|\cdot\|)
	\leq  6 c(\xi) n\eps_n^2(\lambda),
\end{equation}
where $N(3\xi K\eps_n(\lambda), \mathcal F_n(\lambda), \|\cdot\|)$ denotes the minimal number of $L^2(\mathcal X,G)$-balls of radius $3\xi K\eps_n(\lambda)$ required to cover the set $\mathcal F_n(\lambda)$.

We follow the construction in Theorem 2.1 in \cite{vaart2008}. Let us start by introducing the concentration function
\begin{align}\label{def:conc:funct}
	\varphi_{f_0,\lambda}(\eps) := \inf_{\substack{h \in \bb H_\lambda : \\ \|h-f_0\|\leq\eps}} 	\frac12 \|h\|_{\bb H_\lambda}^2 - \log \Pi_\lambda(\|f\| \leq \eps),
\end{align}
where $\bb H_\lambda$ is the reproducing kernel Hilbert space of the prior kernel $k_\lambda$. By Lemma~5.3 in \cite{vaart2008b} and by definition of $\eps_n(\lambda)$, it follows that
\[
	\varphi_{0,\lambda}(K\eps_n(\lambda)) \leq \varphi_{f_0,\lambda}(K \eps_n(\lambda))
	\leq - \log \Pi_\lambda(\|f-f_0\| \leq K \eps_n(\lambda)) = n\eps_n^2(\lambda).
\]
Together with assumption \eqref{e:prior.reg.sc} we obtain the concentration inequality
\begin{equation}\label{e:conc.lambda}
	\varphi_{0,\lambda}(\xi K \eps_n(\lambda)) \leq c(\xi) n\eps_n^2(\lambda).
\end{equation}
Define $D_n = - 2\Phi^{-1}(e^{-c(\xi) n\eps_n^2(\lambda)})$ where $\Phi$ denotes the standard normal CDF and let
\[
	\mathcal F_n(\lambda) = \xi K \eps_n(\lambda) \bb B_1 + D_n \bb H_{\lambda,1},
\]
where $\mathbb B_1$ and $\mathbb H_{\lambda,1}$ denote the unit balls of $L^2(\mathcal X, G)$ and $\bb H_\lambda$, respectively. By the inequality from \cite{borell1975} and \eqref{e:conc.lambda} (and $c(\xi) \geq 1$) it follows that
\begin{align*}
	\Pi_\lambda(\mathcal F_n(\lambda)^c)
	& \leq 1 - \Phi(\Phi^{-1}(e^{-\varphi_{0,\lambda}(\xi K\eps_n(\lambda))}) + D_n) \\
	& \leq 1 - \Phi(-\Phi^{-1}(e^{-n\eps_n^2(\lambda)})) \\
	& = e^{- n\eps_n^2(\lambda)}.
\end{align*}
Furthermore, the argument in the proof of Theorem 2.1 in \cite{vaart2008} shows that \eqref{e:conc.lambda} implies
\[
	\log N(3\xi K \eps_n(\lambda), \mathcal F_n(\lambda),\|\cdot\|) \leq D_n^2/2
	+ \varphi_{0,\lambda}(\xi K \eps_n(\lambda)) \leq 6c(\xi)n\eps_n^2(\lambda),
\]
which is \eqref{e:lambdaEB}.

\paragraph{Testing} Given any hyperparameter $\lambda \in \Lambda_n \setminus \Lambda_{n,0}$ and constant $\xi>0$ small enough we construct a test for $f_0$ versus the set
\[ \mathcal F_n'(\lambda) = \{f \in \mathcal F_n(\lambda) : \|f-f_0\| \geq K\eps_n(\lambda), \|f^{>J_\gamma}\|_\infty \leq \xi K \eps_n(\lambda) \}. \]
From the entropy bound \eqref{e:lambdaEB} we obtain a cover of this set by at most $\exp(6c(\xi) n\eps_n^2(\lambda))$ balls of the form $\{f : \|f-f_i\| \leq 3\xi K \eps_n(\lambda)\}$. By doubling the radii of the balls, we may move the centers $f_i$ inside the set $\mathcal F_n'(\lambda)$, so that in particular $\|f_i-f_0\| \geq K\eps_n(\lambda)$ and $\|f_i^{>J_\gamma}\|_\infty \leq \xi K \eps_n(\lambda)$ for every $i$.

Consider $\bs x$ in the event $\mathcal B_n$ from \eqref{e:Bn} and consider $f$ both in $\mathcal F_n'(\lambda)$ and in one of the balls $\{f : \|f-f_i\| \leq 6\xi K\eps_n(\lambda)\}$ in the aforementioned cover. For $\lambda \in \Lambda_n\setminus \Lambda_{n,0}$ we have $\eps_n(\lambda) \geq \eps_{n,0}$, so we obtain for $\xi = 1/90$ and $K\geq 90$,
\begin{equation}\label{e:f0.tail.eps}
	\|f_0^{>J_\gamma}\|_\infty \leq \delta_n \leq \xi K \eps_n(\lambda).
\end{equation}
It follows that
\begin{align*}
\|f-f_i\|_n
& \leq \|(f-f_i)^{\leq J_\gamma}\|_n + \|f^{>J_\gamma}\|_\infty + \|f_i^{>J_\gamma}\|_\infty \\
& \leq 2 \|f-f_i\| + 2\xi K \eps_n(\lambda) \\
& \leq 14 \xi K \eps_n(\lambda) \\
& = \frac 16 K\eps_n(\lambda) - \xi K \eps_n(\lambda) \\
& \leq \frac 16 \|f_i-f_0\| - \frac 12(\|f_i^{>J_\gamma}\|_\infty + \|f_0^{>J_\gamma}\|_\infty) \\
& \leq \frac 16 \|(f_i-f_0)^{\leq J_\gamma}\| - \frac 13(\|f_i^{>J_\gamma}\|_\infty + \|f_0^{>J_\gamma}\|_\infty) \\
& \leq \frac 13 \|(f_i-f_0)^{\leq J_\gamma}\|_n - \frac 13 \|(f_i-f_0)^{>J_\gamma}\|_n \\
& \leq \frac 13 \|f_i-f_0\|_n.
\end{align*}
This implies in particular that $\|f_i-f_0\|_n \geq \frac{7}{15} K \eps_n(\lambda)$. By Lemma~\ref{l:emp.test} it follows that for every $f_i$ there exists a test $\phi_{n,i}(\lambda)$ such that
\[
	\bs 1_{\mathcal B_n} \bb E_{f_0}[\phi_{n,i}(\lambda) | \bs x]
	\leq e^{-n\|f_0-f_i\|_n^2/(8\sigma^2)} \leq e^{-(K/\sigma)^2 n\eps_n^2(\lambda)/40}
\]
and (from the preceding display)
\begin{align*}
	 \sup_{\substack{f \in \mathcal F_n'(\lambda): \\ \|f-f_i\| \leq 6\xi K\eps_n(\lambda)}}
	\bs 1_{\mathcal B_n} \bb E_f[1-\phi_{n,i}(\lambda) | \bs x] 
	& \leq \sup_{\substack{f \in \mathcal F_n'(\lambda): \\ \|f-f_i\|_n \leq \|f_i-f_0\|_n/3}} \bs 1_{\mathcal B_n} \bb E_f[1-\phi_{n,i}(\lambda) | \bs x] \\
	& \leq \exp(-(K/\sigma)^2n\eps_n^2(\lambda)/360).
\end{align*}
For $K^2 \geq 40\sigma^2(6 c(\xi)+1) \vee 360\sigma^2$, it follows that the test $\phi_n(\lambda) := \max_i \phi_{n,i}(\lambda)$ satisfies, for any $\lambda \in \Lambda_n \setminus \Lambda_{n,0}$,
\begin{equation}\label{e:test.slow.I}
	\bs 1_{\mathcal B_n} \bb E_{f_0}[\phi_n(\lambda) | \bs x] 
	\leq e^{6c(\xi)n\eps_n^2(\lambda) - (K/\sigma)^2 n\eps_n^2(\lambda)/40} 
	\leq e^{-n\eps_{n,0}^2}
\end{equation}
and
\begin{equation}\label{e:test.slow.II}
	\sup_{f \in \mathcal F_n'(\lambda)} \bs 1_{\mathcal B_n} \bb E_f[1-\phi_n(\lambda) | \bs x] \leq e^{-(K/\sigma)^2n\eps_n^2(\lambda)/360} \leq e^{-n\eps_{n,0}^2}.
\end{equation}

By \eqref{e:lambdaPM}, the definition of $\eps_n(\lambda)$ and assumption \eqref{e:priortail}, the remaining prior mass of the set on which we do not test is
\begin{align}
	\nonumber
	 \Pi_\lambda(\mathcal F_n'(\lambda)^c) 
	\nonumber
	& \leq \Pi_\lambda(\mathcal F_n(\lambda)^c) + \Pi_\lambda(\|f-f_0\| \leq K\eps_n(\lambda)) + 	\Pi_\lambda(\|f^{>J_\gamma}\|_\infty \geq \delta_n) \\
	\label{e:rem.slow}
	& \leq e^{-n\eps_n^2(\lambda)} + e^{-n\eps_n^2(\lambda)} + e^{-n\eps_{n,0}^2}.
\end{align}

\paragraph{Completion of the argument} We are now ready to prove \eqref{e:evidenceUB}. Using the tests $\phi_n(\lambda)$ defined above, note that
\begin{align}
\nonumber
& \bb P_{f_0}^{(n)} \Big( \max_{\lambda \in \Lambda_n\setminus\Lambda_{n,0}} \int \frac{p_f^{(n)}}{p_{f_0}^{(n)}}(\bs x,\bs y) \,d\Pi_\lambda(f) > \exp(-n\eps_{n,0}^2/2) \Big) \\
\nonumber
& \leq \bb P_{f_0}^{(n)}(\mathcal B_n^c) + \sum_{\lambda \in \Lambda_n \setminus \Lambda_{n,0}} \bb P_{f_0}^{(n)} \Big(\mathcal B_n \cap \Big\{\int \frac{p_f^{(n)}}{p_{f_0}^{(n)}}(\bs x,\bs y) \,d\Pi_\lambda(f) > \exp(-n\eps_{n,0}^2/2)\Big\} \Big) \\
\label{e:falseReject}
& \leq \bb P_{f_0}^{(n)}(\mathcal B_n^c) + \sum_{\lambda \in \Lambda_n \setminus \Lambda_{n,0}} \bb E_{f_0} \bs 1_{\mathcal B_n} \phi_n(\lambda) \\
\label{e:falseAccept}
& \phantom{\leq{}} + \exp(n\eps_{n,0}^2/2) \sum_{\lambda \in \Lambda_n \setminus \Lambda_{n,0}} \bb E_{f_0} \bs 1_{\mathcal B_n} (1-\phi_n(\lambda)) \int_{\mathcal F_n'(\lambda)} \frac{p_f^{(n)}}{p_{f_0}^{(n)}}(\bs x,\bs y) \,d\Pi_\lambda(f) \\
\label{e:remainingMass}
& \phantom{\leq{}} + \exp(n\eps_{n,0}^2/2) \sum_{\lambda \in \Lambda_n \setminus \Lambda_{n,0}} \bb E_{f_0} \int_{\mathcal F_n'(\lambda)^c} \frac{p_f^{(n)}}{p_{f_0}^{(n)}}(\bs x,\bs y) \,d\Pi_\lambda(f).
\end{align}
We showed earlier that $\mathcal B_n^c$ has probability tending to zero. By \eqref{e:test.slow.I} the sum of type I errors in \eqref{e:falseReject} has upper bound
\[
	|\Lambda_n \setminus \Lambda_{n,0}| \sup_{\lambda \in \Lambda_n\setminus\Lambda_{n,0}}
	\bb E (\bs 1_{\mathcal B_n} \bb E_{f_0}[\phi_n(\lambda) | \bs x])
	\leq |\Lambda_n| \exp(-n\eps_{n,0}^2).
\]
This term vanishes by the assumption on the size of $\Lambda_n$. Similarly \eqref{e:test.slow.II} implies that the sum in \eqref{e:falseAccept} is bounded from above by
\begin{align*}
&\sum_{\lambda \in \Lambda_n \setminus \Lambda_{n,0}}\bb E \Big( \bs 1_{\mathcal B_n} \bb E_{f_0}\Big[(1-\phi_n(\lambda)) \int_{\mathcal F_n'(\lambda)} \frac{p_f^{(n)}}{p_{f_0}^{(n)}}(\bs y|\bs x) \,d\Pi_\lambda(f) \Big| \bs x \Big] \Big) \\
& \leq \sum_{\lambda \in \Lambda_n \setminus \Lambda_{n,0}}\bb E \int_{\mathcal F_n'(\lambda)} \bs 1_{\mathcal B_n} \bb E_f[1-\phi_n(\lambda)|\bs x] \,d\Pi_\lambda(f) \\
& \leq |\Lambda_n| \exp(-n\eps_{n,0}^2),
\end{align*}
and so the full term \eqref{e:falseAccept} vanishes. Finally, using \eqref{e:rem.slow}, the term in \eqref{e:remainingMass} is bounded from above by
\begin{align*}
\exp(n\eps_{n,0}^2/2) |\Lambda_n\setminus\Lambda_{n,0}| \max_{\lambda \in \Lambda_n\setminus\Lambda_{n,0}} \Pi_\lambda(\mathcal F_n'(\lambda)^c) \lesssim |\Lambda_n| \exp(-n\eps_{n,0}^2/2)
\end{align*}
which also vanishes.
\end{proof}

\subsection{Contraction rate for the empirical Bayes posterior}
Let us recall the definition of the empirical Bayes posterior given in \eqref{def:eb}, which is attained by plugging in the MMLE to the posterior distribution. This empirical approach is often used in practice as a natural and computationally more convenient data-driven tuning instead of the hierarchical Bayes method. The theorem below is the adaptation of Theorem 2.2 of \cite{rousseau2017} to our setting. Similarly to Lemma \ref{l:mmleSet}, several technical issues had to be overcome when relating the empirical and standard $L^2$-norms, hence we provide the main steps of the proof below.
\begin{theorem}\label{t:ebp.contr}
Under the conditions of Lemma~\ref{l:mmleSet}, there exists $M>0$ large enough such that
\[
	\bb E_{f_0}\Pi_{\hat\lambda_n} (\|f-f_0\|>M \eps_{n,0} | \bs x, \bs y) \to 0.
\]
\end{theorem}

\begin{proof}[Proof of Theorem~\ref{t:ebp.contr}]
We slightly adapt the prior mass and testing argument from the proof of Lemma~\ref{l:mmleSet} to accomodate $\lambda \in \Lambda_{n,0}$. For any such $\lambda$ we define
\[
	\mathcal F_n(\lambda) = \xi K \eps_n(\lambda) \bb B_1 + D_n \bb H_{\lambda, 1}
\]
where $D_n = -2\Phi^{-1}(e^{-c(\xi)n\eps_{n,0}^2})$. Here the concentration inequality is $\varphi_{0,\lambda}(\xi K \eps_n(\lambda)) \leq c(\xi) n\eps_{n,0}^2$, and the argument in the proof of Lemma~\ref{l:mmleSet} gives
\begin{equation}\label{e:ent.fast}
	\log N(\xi K \eps_n(\lambda), \mathcal F_n(\lambda), \|\cdot\|)
	\leq 6 c(\xi) n\eps_{n,0}^2
\end{equation}
and
\begin{equation}\label{e:lambdaPM.fast}
	\Pi_\lambda(\mathcal F_n(\lambda)^c) \leq e^{-n\eps_{n,0}^2}.
\end{equation}

For $\lambda \in \Lambda_{n,0}$ we construct a test $\phi_n(\lambda)$ for $f_0$ versus the set
\[
	\mathcal F_n'(\lambda) = \{ f \in \mathcal F_n(\lambda) 
	: \|f-f_0\| > M\eps_{n,0}, \|f^{>J_\gamma}\|_\infty \leq \delta_n \}.
\]

Here we note that any $f \in \mathcal F_n'(\lambda)$ with $\|f-f_i\| \leq 6\xi M \eps_n(\lambda)$ satisfies on $\mathcal B_n$, for $\xi = 1/90$ and $M \geq 90$,
\begin{align*}
\|f-f_i\|_n
& \leq \|(f-f_i)^{\leq J_\gamma}\|_n + \|f^{>J_\gamma}\|_\infty + \|f_i^{>J_\gamma}\|_\infty \\
& \leq 2 \|f-f_i\| + 2 \delta_n \\
& \leq 12 \xi M \eps_{n,0} + 2 \xi M \delta_n \\
& \leq \frac 16 M \eps_{n,0} - \delta_n \\
& \leq \frac 16 \|f_i-f_0\| - \frac 12(\|f_i^{>J_\gamma}\|_\infty + \|f_0^{>J_\gamma}\|_\infty) \\
& \leq \frac 16 \|(f_i-f_0)^{\leq J_\gamma}\| - \frac 13(\|f_i^{>J_\gamma}\|_\infty + \|f_0^{>J_\gamma}\|_\infty) \\
& \leq \frac 13 \|(f_i-f_0)^{\leq J_\gamma}\|_n - \frac 13 \|(f_i-f_0)^{>J_\gamma}\|_n \\
& \leq \frac 13 \|f_i-f_0\|_n.
\end{align*}
Covering the set $\mathcal F_n'(\lambda)$ with at most $\exp(6 c(\xi) n\eps_{n,0}^2)$ balls and taking $M$ sufficiently large, it follows that there exists a test $\phi_n(\lambda)$ with
\begin{equation}\label{e:test.fast.I}
	\bs 1_{\mathcal B_n}\bb E_{f_0}[\phi_n(\lambda) | \bs x] \leq e^{-n\eps_{n,0}^2}
\end{equation}
and
\begin{equation}\label{e:test.fast.II}
	\sup_{f \in \mathcal F_n'(\lambda)} \bs 1_{\mathcal B_n} \bb E_f[1-\phi_n(\lambda) | \bs x] 	\leq e^{-n\eps_{n,0}^2}.
\end{equation}

Let us define the test $\phi_n' = \max_{\lambda \in \Lambda_{n,0}} \phi_n(\lambda)$. Note that by \eqref{e:inLambda0} and \eqref{e:evidenceLB}, the event 
\begin{equation}\label{e:eventCn}
\mathcal C_n := \mathcal B_n \cap \{\hat\lambda_n \in \Lambda_{n,0}\} \cap \Big\{\int \frac{p_f^{(n)}}{p_{f_0}^{(n)}}(\bs x,\bs y) \,d\Pi_{\hat\lambda_n}(f) > \exp(-C_1 \min_{\lambda \in \Lambda_n} n\eps_n^2(\lambda)) \Big\}
\end{equation}
has probability tending to 1 (where $C_1 = 1 + 2K^2/\sigma^2$ as in the proof of Lemma~\ref{l:mmleSet}). Then, by Bayes' rule,
\begin{align}
\nonumber
& \bb P_{f_0}^{(n)} \Pi_{\hat\lambda_n} (\|f-f_0\| > M \eps_{n,0} | \bs x, \bs y) \\
\nonumber
& \leq \bb P_{f_0}^{(n)}(\mathcal C_n^c) + \bb E_{f_0} \bs 1_{\mathcal C_n} \phi_n' \\
\nonumber
& \phantom{\leq{}} + \bb E_{f_0} \bs 1_{\mathcal C_n} (1-\phi_n') \frac{\int_{\|f-f_0\| > M\eps_{n,0}} (p_f^{(n)}/p_{f_0}^{(n)})(\bs x,\bs y) \,d\Pi_{\hat\lambda_n}(f)}{\int(p_f^{(n)}/p_{f_0}^{(n)})(\bs x,\bs y)\,d\Pi_{\hat\lambda_n}(f)} \\
\nonumber
& \leq o(1) + \bb E_{f_0}\bs 1_{\mathcal C_n} \phi_n' \\
\label{e:posteriorBoundEB}
& \phantom{\leq{}} + e^{C_1 \min_{\lambda\in\Lambda_n} n\eps_n^2(\lambda)} \sum_{\lambda \in \Lambda_{n,0}} \bb E_{f_0} \bs 1_{\mathcal C_n} (1-\phi_n(\lambda)) \int_{\|f-f_0\| > M\eps_{n,0}} \frac{p_f^{(n)}}{p_{f_0}^{(n)}}(\bs x,\bs y) \,d\Pi_\lambda(f). 
\end{align}
The type I error can be bounded using \eqref{e:test.fast.I} as
\[
	\bb E_{f_0} \bs 1_{\mathcal C_n} \phi_n'
	\leq \sum_{\lambda \in \Lambda_{n,0}} \bb E \bs 1_{\mathcal B_n}
	\bb E_{f_0}[\phi_n(\lambda) | \bs x]
	\leq  |\Lambda_{n,0}|\exp(-n\eps_{n,0}^2),
\]
which vanishes due to the assumption on $|\Lambda_n| \geq |\Lambda_{n,0}|$. To bound the remainder in \eqref{e:posteriorBoundEB}, in view of \eqref{e:test.fast.II}, \eqref{e:lambdaPM.fast}, and condition \eqref{e:priortail} we obtain
\begin{align*}
& \bb E \int_{\mathcal F_n'(\lambda)} \bs 1_{\mathcal C_n} \bb E_f[1-\phi_n(\lambda)|\bs x] \,d\Pi_\lambda(f) + \Pi_\lambda(\mathcal F_n(\lambda)^c) + \Pi_\lambda(\|f^{>{J_\gamma}}\|_\infty > \delta_n) \\
& \leq 3\exp(-n\eps_{n,0}^2).
\end{align*}
Hence, the full last term in \eqref{e:posteriorBoundEB} is bounded by $3|\Lambda_n| \exp(-n\eps_{n,0}^2/2)=o(1)$ for $n$ large enough, concluding the proof of our statement.
\end{proof}

\subsection{Proof of Theorem \ref{t:hbp.contr}}\label{proof:t:hbp.contr}

The proof extends the results on the empirical Bayes posterior to the hierarchical Bayes method, see for instance \cite{szabo2013,knapik2016bayes,rousseau2017} for similar derivations. One of the key difference is the explicit exponential upper bound for the posterior contraction, needed to obtain frequentist guarantees for the variational approximation.

Let us split the probability to the cases when $\lambda\in \Lambda_{n,0}$ and $\lambda\in \Lambda_n \setminus \Lambda_{n,0}$, i.e.
\begin{multline}\label{e:hp.bound}
\Pi(\|f-f_0\| \geq M \eps_{n,0} | \bs x,\bs y) \leq \Pi(\lambda \in \Lambda_n\setminus\Lambda_{n,0} | \bs x,\bs y) \\
+ \sum_{\lambda \in \Lambda_{n,0}} \pi(\lambda | \bs x,\bs y) \Pi_\lambda(\|f-f_0\| \geq M \eps_{n,0} | \bs x,\bs y).
\end{multline}
We first bound the top term. By \eqref{e:evidenceLB0}, with the particular $\lambda_{n,0}$ such that $\eps_n(\lambda_{n,0}) = \min_\lambda \eps_n(\lambda)$, we have
\begin{equation}\label{e:elb.lambda0}
\bb P_{f_0}^{(n)}\Big( \int \frac{p_f^{(n)}}{p_{f_0}^{(n)}}(\bs x,\bs y) \,d\Pi_{\lambda_{n,0}}(f) > \exp(-C_1 \min_{\lambda \in \Lambda_n} n\eps_n^2(\lambda)) \Big) \to 1.
\end{equation}
Likewise, \eqref{e:evidenceUB} gives
\begin{equation}\label{e:eub.bad.lambda}
	\bb P_{f_0}^{(n)}\Big( \max_{\lambda \in \Lambda_n\setminus\Lambda_{n,0}} \int \frac{p_f^{(n)}}{p_{f_0}^{(n)}}(\bs x,\bs y) \,d\Pi_\lambda(f) \leq \exp(-n\eps_{n,0}^2/2)\Big) \to 1.
\end{equation}
On the intersection of the two events we have
\begin{align*}
\Pi(\lambda \in \Lambda_n\setminus\Lambda_{n,0} | \bs x,\bs y)
& \leq \frac{\sum_{\lambda \in \Lambda_n\setminus\Lambda_{n,0}}\pi(\lambda)\int (p_f^{(n)}/p_{f_0}^{(n)})(\bs x,\bs y) \,d\Pi_\lambda(f) }{\pi(\lambda_{n,0})\int (p_f^{(n)}/p_{f_0}^{(n)})(\bs x,\bs y) \,d\Pi_{\lambda_{n,0}}(f)} \\
& \leq \exp(-n\eps_{n,0}^2/2 - \log \pi(\lambda_{n,0}) + C_1n\eps_{n,0}^2/c_n^2) \\
& \leq \exp(-n\eps_{n,0}^2/3)
\end{align*}
for $n$ large enough, where we used Bayes' rule and the assumption \eqref{e:hyper.mass}.

Regarding the bottom term in \eqref{e:hp.bound}, the proof of Theorem~\ref{t:ebp.contr} gives $M>0$ such that
\[
	\max_{\lambda\in\Lambda_{n,0}} \bb E_{f_0} \Pi_\lambda(\|f-f_0\| \geq M \eps_{n,0} | \bs x,\bs y) \bs 1_{\mathcal C_n} 
	\leq 4 e^{-n\eps_{n,0}^2/2}
\]
for all $n$ large enough, where $\mathcal C_n$ is the event given in \eqref{e:eventCn}. Bounding each of the posterior probabilities $\pi(\lambda | \bs x,\bs y)$ by one, it follows that the final term in \eqref{e:hp.bound} is bounded from above by $4 |\Lambda_n| e^{-n\eps_{n,0}^2/2}$.

Taking $\mathcal A_n$ as the intersection of $\mathcal C_n$ with the events in \eqref{e:elb.lambda0} and \eqref{e:eub.bad.lambda}, we have $\bb P_{f_0}^{(n)}(\mathcal A_n) \to 1$ and
\[ \bb E_{f_0} \Pi(\|f-f_0\| \geq M \eps_{n,0} | \bs x,\bs y)\bs 1_{\mathcal A_n}  \leq 5 e^{-n\eps_{n,0}^2/3} \]
for $n$ large enough, which completes the proof.

\subsection{Proof of Theorem \ref{t:vbp.contr}}\label{proof:t:vbp.contr}
In view of Theorem~\ref{t:hbp.contr} there exists a $M>0$ such that
\begin{equation}\label{e:hier.contr}
	\bb E_{f_0} \Pi(\|f-f_0\| \geq M \eps_{n,0} | \bs x,\bs y) \bs 1_{\mathcal A_n}
	\lesssim \exp(-n\eps_{n,0}^2/3)
\end{equation}
for some sequence of events $\mathcal A_n$, with $\bb P_{f_0}^{(n)}(\mathcal A_n) \to 1$. We replace $\mathcal A_n$ by its intersection with the event
\begin{equation}\label{e:lambda.n.elb}
	\Big\{\int \frac{p_f^{(n)}}{p_{f_0}^{(n)}}(\bs x,\bs y) \,d\Pi_{\lambda_n}(f)
	\geq e^{-C_1n\eps_n^2(\lambda_n)}\Big\} \cap \Big\{\int \frac{p_f^{(n)}}{p_{f_0}^{(n)}}(\bs x,\bs y) \,d\Pi(f) \leq e^{n\eps_n^2(\lambda_n)} \Big\}.
\end{equation}
The resulting event still has probability tending to $1$ by \eqref{e:lambdaELB} and Markov's inequality. Then in view of Theorem 5 of \cite{ray2022},
\begin{multline}\label{e:var.post.c}
	\bb E_{f_0} \tilde \Pi_{\lambdavar}(\|f-f_0\| \geq M\eps_{n,0} |\bs x,\bs y) \lesssim \\
	\frac{1}{n\eps_{n,0}^2} \bb E_{f_0} \bs 1_{\mathcal A_n}
	\KL(\tilde\Pi_{\lambdavar}(\,\cdot\, |\bs x,\bs y) \| \Pi(\,\cdot\,| \bs x,\bs y)) + o(1).
\end{multline}

Note that for $Q$ as given in the statement of the theorem,
\begin{align*}
	\KL(\tilde \Pi_{\lambdavar}(\,\cdot\,|\bs x,\bs y) \| \Pi(\,\cdot\,|\bs x,\bs y))
	& \leq \KL(Q \| \Pi(\,\cdot\,|\bs x,\bs y)) \\
	& \leq \int \log\Big(\frac{1}{\pi(\lambda_n | \bs x,\bs y)} \frac{dQ}{d\Pi_{\lambda_n}
	(\,\cdot\,|\bs x,\bs y)} \Big) \,d Q \\
	& = \KL(Q \| \Pi_{\lambda_n}(\,\cdot\,|\bs x,\bs y)) - \log \pi(\lambda_n | \bs x,\bs y).
\end{align*}
Hence by assumption \eqref{e:vbcCondition} it follows that \eqref{e:var.post.c} tends to zero as soon as
\begin{equation}\label{e:post.pi}
	- \bb E_{f_0} \bs 1_{\mathcal A_n} \log \pi(\lambda_n | \bs x,\bs y) = o(n\eps_{n,0}^2).
\end{equation}
Bayes' rule gives
\[
	\pi(\lambda_n | \bs x,\bs y) = \frac
	{\pi(\lambda_n) \int (p_f^{(n)}/p_{f_0}^{(n)})(\bs x,\bs y) \,d\Pi_{\lambda_n}(f)}
	{\int (p_f^{(n)}/p_{f_0}^{(n)})(\bs x,\bs y) \,d\Pi(f)},
\]
so by assumption \eqref{e:hyper.mass} and since $\mathcal A_n$ includes the event \eqref{e:lambda.n.elb}, the left-hand side of \eqref{e:post.pi} is bounded by
\[ o(n\delta_n^2) + C_1 n\eps_n^2(\lambda_n) + n\eps_n^2(\lambda_n) = o(n\eps_{n,0}^2), \]
concluding the proof of the theorem.

\section{Proofs for the examples}\label{sec:proof:examples}
In this section we provide the proofs for the contraction rates for the specific choices of Gaussian process priors and variational approaches using Theorem~\ref{t:vbp.contr}. But before that we introduce some notations.

The reproducing kernel Hilbert space $\bb H_\lambda$ associated to a GP with series expansion \eqref{def:GP:prior} is the subspace of $L^2(\mathcal X, G)$ consisting of those functions $g$ such that 
\[ \sum_{j=1}^\infty |\ip{g,\varphi_j}|^2 / (s_j^\lambda)^2 < \infty. \]
This quantity is the squared RKHS norm $\|g\|_{\bb H_\lambda}^2$. 

In the examples below we compute $\bar\eps_n(\lambda)$ by solving the inequality $\varphi_{f_0,\lambda}(\bar\eps_n(\lambda)) \leq n\bar\eps_n(\lambda)^2$, where $\varphi_{f_0,\lambda}$ is the concentration function defined in \eqref{def:conc:funct}. By Lemma 5.3 in \cite{vaart2008b}, it follows that (for $K \geq 2$)
\[ \Pi_\lambda(\|f-f_0\| \leq K \bar\eps_n(\lambda)) \geq e^{-n\bar\eps_n^2(\lambda)}, \]
which implies $\bar\eps_n(\lambda)$ is an upper bound for $\eps_n(\lambda)$ (as defined in \eqref{e:eps.lambda}). Similarly, if $\underline\eps_n(\lambda)$ solves $\varphi_{0,\lambda}(K\underline\eps_n(\lambda)) \geq n\underline\eps_n^2(\lambda)$, it is a lower bound for $\eps_n(\lambda)$.  

In each corollary we verify the conditions of Theorem~\ref{t:vbp.contr}. First we show the tail condition \eqref{e:f0.tail} on the signal $f_0 \in \mathcal S^\beta(L)$ holds with $\delta_n = c_n n^{-\beta/(d+2\beta)}$ for sufficiently smooth signals. Note that for $\beta \geq d+2\gamma$  the Cauchy-Schwarz inequality gives
\begin{align*}
	\|f_0^{>J_\gamma}\|_\infty
	& \lesssim \sum_{j>J_\gamma} j^{\beta/d} |\ip{f_0,\varphi_j}| j^{(\gamma-\beta)/d} \\
	& \leq \Big(\sum_{j>J_\gamma} j^{2(\gamma-\beta)/d}
	\sum_{j>J_\gamma} j^{2\beta/d} \ip{f_0,\varphi_j}^2\Big)^{1/2} \\
	& = o(J_\gamma^{(\gamma-\beta)/d+1/2}) =o(\delta_n).
\end{align*}
Hence this lower bound for $\beta$ will always be assumed.

\subsection{Proof of Corollary \ref{cor:poly}}\label{sec:cor:poly}
We verify the remaining conditions of Theorem~\ref{t:vbp.contr}. Here the smoothness parameter $\lambda = \alpha$ is selected empirically. First note that in view of Corollary 4.3 in \cite{dunker1998} there exists a constant $C > 0$ independent of $\alpha\in \Lambda^{(n)}$ such that for $\eps>0$ small enough
\begin{equation}\label{e:poly.sb}
	C^{-1} \eps^{-d/\alpha} \leq -\log \Pi_\alpha(\|f\|\leq\eps) = \varphi_{0,\alpha}(\eps) \leq C \eps^{-d/\alpha}.
\end{equation}
Recall that we have assumed $f_0 \in \mathcal S^\beta(L)$ so $\sum j^{2\beta/d}\ip{f_0,\varphi_j}^2 \leq L^2$. Letting $h = \sum_{j \leq J} \ip{f_0,\varphi_j} \varphi_j \in \bb H_\alpha$, note that
\[
	\|h-f_0\| = \Big(\sum_{j>J} j^{-2\beta/d} j^{2\beta/d} \ip{f_0,\varphi_j}^2\Big)^{1/2} \leq J^{-\beta/d} L.
\]
With $J = \lceil (L/\eps)^{d/\beta} \rceil$ it follows that for all sufficiently small $\eps>0$
\[
	\inf_{\substack{h \in \bb H_\alpha : \\ \|h-f_0\|\leq\eps}} \|h\|_{\bb H_\alpha}^2
	\leq \sum_{j \leq J} j^{1+2\alpha/d-2\beta/d} j^{2\beta/d} \ip{f_0,\varphi_j}^2
	\leq 2 (1 \vee J)^{1+2\alpha/d-2\beta/d} L^2,
\]
so the concentration function can be bounded for $f_0 \in \mathcal S^\beta(L)$ as
\[
	\varphi_{f_0,\alpha}(\eps) \lesssim 1 \vee \eps^{(2\beta-2\alpha-d)/\beta} + \eps^{-d/\alpha},
\]
which implies $\eps_n(\alpha) \lesssim n^{-(\alpha \wedge \beta)/(d+2\alpha)}$. The $\alpha \in \Lambda^{(n)}$ that maximises $(\beta\wedge\alpha)/(d+2\alpha)$ has at most distance $1/\log n$ to $\beta$, so it satisfies
\[
	\frac{\alpha\wedge\beta}{d+2\alpha} \geq \frac{\beta-1/\log n}{d+2\beta}
	\wedge \frac{\beta}{d+2(\beta+1/\log n)} = \frac{\beta-1/\log n}{d+2\beta}.
\]
We take $\lambda_n$ to be this particular $\alpha$. Consequently,
\[
	\min_{\alpha \in \Lambda^{(n)}} \eps_n(\alpha) \lesssim
	\min_{\alpha \in \Lambda^{(n)}} n^{-(\alpha \wedge \beta)/(d+2\alpha)}
	\leq e n^{-\beta/(d+2\beta)}.
\]
Since $\delta_n = c_n n^{-\beta/(d+2\beta)}$ it follows that the contraction rate is
\[ \eps_{n,0} \lesssim c_n n^{-\beta/(d+2\beta)} \]
as soon as we verify the remaining conditions of Theorem~\ref{t:vbp.contr} for the current choice of prior and variational scheme.

The lower bound in \eqref{e:poly.sb} also gives $\varphi_{0,\lambda}(K\underline \eps_n(\alpha)) \geq n\underline\eps_n^2(\alpha)$ for $\underline\eps_n(\alpha)$ a small multiple of $n^{-\alpha/(d+2\alpha)}$, and it follows that $\min_{\alpha\in\Lambda_n} n\eps_n^2(\alpha)\to\infty$.

Now let us consider condition \eqref{e:priortail}. We use a technique similar to the proof of Lemma 7 in \cite{randrianarisoa2023}. We have already bounded $\eps_{n,0} \lesssim c_n n^{-\beta/(d+2\beta)}$ hence $n\eps_{n,0}^2 \lesssim c_n^2 n^{d/(d+2\beta)}$. Since $J_\gamma = (n/\log^2 n)^{d/(d+2\gamma)}$ and $\beta > \gamma$ it follows that for $c_n\to\infty$ not too fast, $n\eps_{n,0}^2 = o(J_\gamma)$. For such a $c_n$, by the assumptions on $f = \sum_{j=1}^\infty s_j^\alpha Z_j \varphi_j$, there exists $C>0$ such that (for $n$ large enough)
\begin{align*}
	\|f^{>J_\gamma}\|_\infty
	& \leq C \sum_{j>J_\gamma} j^{\gamma/d-\alpha/d-1/2} |Z_j| \\
	& \leq C \sum_{\substack{k \in \bb N: \\ k \ge J_\gamma/(n\eps_{n,0}^2)}} (kn\eps_{n,0}^2)^{\gamma/d-\alpha/d-1/2} \sum_{\substack{j \in \bb N: \\ kn\eps_{n,0}^2 < j \le (k+1)n\eps_{n,0}^2}} |Z_j|.
\end{align*}
Therefore, by the Chernoff bound with $t>0$, since $\alpha > \gamma+d/2$, it holds that for some $C'>0$
\begin{align*}
	\Pi_\alpha(\|f^{>J_\gamma}\|_\infty \ge \delta_n)
	& \leq e^{-t\delta_n} \bb E \exp\Big(tC\sum_{k \geq J_\gamma/(n\eps_{n,0}^2)}
	(kn\eps_{n,0}^2)^{\gamma/d-\alpha/d-1/2} \sum_{j=1}^{\lceil n\eps_{n,0}^2 \rceil}|Z_j|\Big) \\
	& \leq e^{-t\delta_n} \Big(\bb E \exp(tC' (n\eps_{n,0}^2)^{-1}J_\gamma^{\gamma/d-\alpha/d+1/2} |Z_1|)\Big)^{\lceil n\eps_{n,0}^2 \rceil}.
\end{align*}
Using $\bb E e^{s|Z_1|} \leq 2 \bb E e^{s Z_1} = 2 e^{s^2/2}$ it follows that for some $C''>0$,
\[
	\Pi_\alpha(\|f^{>J_\gamma}\|_\infty \ge \delta_n) \leq \exp(-t\delta_n + \log(2)n\eps_{n,0}^2 + t^2 C'' (n\eps_{n,0}^2)^{-1} J_\gamma^{2\gamma/d-2\alpha/d+1}).
\]
Optimising over $t$, it follows that the term in the exponent is of the order
\[ - n\eps_{n,0}^2 \cdot \delta_n^2 J_\gamma^{2\alpha/d-2\gamma/d-1} \]
where $\delta_n^2 J_\gamma^{2\alpha/d-2\gamma/d-1} \to \infty$ since $\alpha > d+2\gamma$. Hence \eqref{e:priortail} is satisfied for $c_n$ as given.

Next, by \eqref{e:poly.sb}, for arbitrary $\xi\in(0,1)$ and $\eps>0$ small enough
\begin{equation}\label{e:poly.small.ball}
-\log\Pi_\alpha(\|f\|\leq \xi \eps) \leq - C^{2}\xi^{-d/\alpha} \log\Pi_\alpha(\|f\|\leq\eps),
\end{equation}
verifying condition \eqref{e:prior.reg.sc} for $\alpha \in [\beta^-,\beta^+]$. 

Condition \eqref{e:hyper.mass} is satisfied because the hyper-prior is uniform on a set of cardinality $|\Lambda_n| \leq |\beta^+-\beta^-| \log n$. 

~

It remains to verify the condition \eqref{e:vbcCondition} on the KL-divergence. Recall that $\overline\eps_n(\lambda)$ is an upper bound for $\eps_n(\lambda)$ computed by solving the concentration inequality $\varphi_{f_0,\lambda}(\overline\eps_n(\lambda)) \leq n\overline\eps_n^2(\lambda)$. Consider a $\lambda_n \in \Lambda_n$ which is at most distance $1/\log n$ away from $\beta$. It was already shown above that $\eps_n(\lambda_n) \leq \overline \eps_n(\lambda_n) \lesssim n^{-\beta/(d+2\beta)}$. The concentration inequality implies that there exists $h \in \bb H_{\lambda_n}$ such that $\|f_0-h\| \leq \overline\eps_n(\lambda_n)$ and $\|h\|_{\bb H_{\lambda_n}}^2 \leq n\overline\eps_n^2(\lambda_n)$. Combining this with Lemma 3 in \cite{vbgp} it follows that
\begin{multline*}
	\bb E_{f_0} \KL(\tilde\Pi_{\lambda_n}(\,\cdot\,|\bs x,\bs y) \|
	\Pi_{\lambda_n}(\,\cdot\,|\bs x,\bs y)) \lesssim \\
	(1 + \bb E_{\bs x}\|K_\ff-Q_\ff\|_{\mathrm{op}})n\bar\eps_n^2(\lambda_n)
	 + \bb E_{\bs x} \tr(K_\ff-Q_\ff).
\end{multline*}
Hence it suffices to show that
\begin{equation}\label{eq:opnorm}
	\bb E_{\bs x} \|K_\ff - Q_\ff\|_{\mathrm{op}} \lesssim 1
\end{equation}
and
\begin{equation}\label{eq:trace}
	\bb E_{\bs x} \tr(K_\ff-Q_\ff) \lesssim n \cdot n^{-2\beta/(d+2\beta)} = n^{d/(d+2\beta)}.
\end{equation}

In view of Lemma 4 and 5 of \cite{vbgp} the bounds \eqref{eq:opnorm} and \eqref{eq:trace} are satisfied for the variational classes studied here with polynomially decaying prior eigenvalues and number of spectral features exceeding $m \geq n^{d/(d+2\lambda_n)} \asymp n^{d/(d+2\beta)}$. However, these bounds were shown under the assumption that the eigenfunctions of the prior are bounded by a constant. This means that for the population spectral features we have to slightly adapt the argument, since the assumption on the eigenfunctions is weakened to \eqref{e:eigBound} in the present paper. For the empirical features the result does not rely on this bound and is hereby proved.

For the population spectral features, Hoeffding's lemma gives
\begin{equation}\label{e:hoef.induc}
	\bb P_x\Big( |\bs\varphi_\ell^\T \bs\varphi_k - n\delta_{\ell k}| \geq C \sqrt{n^{1+2\gamma/\alpha}\log n} \Big)
	\leq 2 \exp\Big(\frac{-2C^2 n^{1+2\gamma/\alpha}\log n}{n (2C_\varphi^2(\ell k)^{\gamma/d})^2} \Big),
\end{equation}
hence by a union bound,
\[
	\bb P_x\Big( \max_{\ell,k \leq n^{d/(2\alpha)}}
	|\bs\varphi_\ell^\T \bs\varphi_k - n\delta_{\ell k}|
	\geq C \sqrt{n^{1+2\gamma/\alpha}\log n} \Big)
	\leq 2 n^{d/\alpha} n^{-C^2/(2C_\varphi^4)}.
\]
Correspondingly, for $\alpha > d+2\gamma$, the proof of Lemma 5 in \cite{vbgp} gives
\[
	\bb E_{\bs x} \|K_\ff - Q_\ff\|_{\mathrm{op}}
	\lesssim 1 + n m^{-1-2\alpha/d} + n^{d/(2\alpha) + 2\gamma/\alpha}m^{-2\alpha/d}\log n,
\]
and
\[
	\bb E_{\bs x} \tr(K_\ff-Q_\ff) \lesssim nm^{-2\alpha/d}.
\]
Taking $\alpha = \lambda_n$, the element of $\Lambda_n$ closest to $\beta$, we recall that $\eps_n(\lambda_n) \lesssim n^{-2\beta/(d+2\beta)}$. For $m \geq n^{d/(d+2\alpha)}$ and $\alpha > d+2\gamma$ it follows that \eqref{eq:opnorm} and \eqref{eq:trace} are satisfied. Hence for the population spectral features the condition on the KL-divergence in Theorem~\ref{t:vbp.contr} is satisfied which gives
\[ \bb E_{f_0} \tilde\Pi_{\hat\alpha_n^{\mathrm{var}}}(\|f-f_0\| \geq M\eps_{n,0} | \bs x,\bs y) \to 0, \]
thereby concluding the proof.

\subsection{Proof of Corollary \ref{cor:exp}}\label{sec:cor:exp}

We first determine an upper bound for $\eps_n(\tau)$. In view of Lemma \ref{l:tauRKHS} and \ref{l:tauSB} given below, for $f_0 \in \mathcal S^\beta(L)$, the concentration function \eqref{def:conc:funct} is bounded from above by  
\[
	\varphi_{f_0,\tau}(\eps) \leq L^2 \tau^{-d} e^{\tau} + \tau^{-d} \eps^2 \exp(2\tau(L/\eps)^{1/\beta}) + C \tau^{-d} (\log \tfrac1\eps)^{d+1}
\]
for some universal constant $C>0$. If $\eps \geq 2Le^{\tau/2}/\sqrt{n\tau^d}$ then
\[
	L^2 \tau^{-d}e^\tau \leq n\eps^2/3.
\]
Taking $\eps \geq D (\tau/\log n)^\beta$ for sufficiently large constant $D>0$, in view of $\tau^{-d}=O(n^{\frac{d}{d+2\beta^{-}}})$, we have for any $\tau \in \Lambda_n$
\[
	\tau^{-d} \eps^2 \exp(2\tau(L/\eps)^{1/\beta}) \leq n\eps^2/3,
\]
while for $\eps \geq D' (n\tau^d)^{-1/2} (\log n)^{(d+1)/2}$ with $D'>0$ large enough, it follows that
\[
	C \tau^{-d}(\log \tfrac 1\eps)^{d+1} \leq n\eps^2/3.
\]
Hence the concentration function inequality $\varphi_{f_0,\tau}(\eps)\leq n\eps^2$ is solved by $\eps$ equal to the maximum of the three preceding lower bounds, which gives
\[
	\eps_n(\tau) \lesssim e^{\tau/2}/\sqrt{n\tau^d} \vee (\tau / \log n)^\beta \vee (n\tau^d)^{-1/2}(\log n)^{(d+1)/2}.
\]
The expression on the right is minimised for
\begin{equation}\label{e:opt.tau}
	\tau = c n^{-1/(d+2\beta)} (\log n)^{1+1/(d+2\beta)},
\end{equation}
(over the set $\Lambda_n$ we obtain this number up to at most a factor $e$) hence
\[
	\eps_{n,0} \lesssim c_n (n/\log n)^{-\beta/(d+2\beta)}.
\]

As before any solution to the reverse concentration inequality $\varphi_{0,\tau}(K\underline\eps_n) \geq n\underline\eps_n^2$ gives a lower bound for $\eps_n(\tau)$. In this case Lemma~\ref{l:tauSB} shows that $\underline \eps_n$ equal to a small multiple of $\sqrt{(\log n)/n}$ works for any $\tau \in (0,1)$. It follows that $\min_{\tau\in\Lambda_n} n\eps_n^2(\tau) \geq \log n \to \infty$. 

To verify condition \eqref{e:priortail}, we show first that the exponentially decaying coefficients $s_j^\tau$ are bounded up to a constant by the polynomially decaying coefficients from the other example. A linear approximation of the function $x \mapsto e^{\tau x/(d+2\alpha)}$ at $x = J_\gamma^{1/d}/2$ gives
\begin{align*}
	e^{\tau j^{1/d}/(d+2\alpha)}
	& \geq \frac{\tau}{2d+4\alpha}e^{\tau J_\gamma^{1/d}/(2d+4\alpha)}(j^{1/d}-J_\gamma^{1/d}/2)
	 + e^{\tau J_\gamma^{1/d}/(2d+4\alpha)} \\
	& \geq \frac{\tau}{4d+8\alpha} e^{\tau J_\gamma^{1/d}/(2d+4\alpha)} j^{1/d}
\end{align*}
for $j \geq J_\gamma$ and $\alpha>0$. Recalling that $J_\gamma = (n/(\log n)^2)^{d/(d+2\gamma)}$ and $\tau \geq \min \Lambda_n = n^{-1/(d+2\beta^-)}$ where $\beta^- > d+2\gamma$, it follows that the term $\tau J_\gamma^{1/d}$ in the exponent is bounded from below by $n^{c}$, for some $c>0$. Therefore the term $\frac{\tau}{4d+8\alpha} e^{\tau J_\gamma^{1/d}/(2d+4\alpha)}$ is uniformly bounded away from 0 for $\tau\in \Lambda_n$, $n\in\mathbb{N}$, and the preceding display implies that there exists a positive constant $C$ such that 
\[ e^{-\tau j^{1/d}} \leq C j^{-1-2\alpha/d} \]
for all $n \in \bb N$, $j \geq J_\gamma$ and $\tau \in \Lambda_n$. Hence the tail bound \eqref{e:priortail} follows from its counterpart for polynomially decaying eigenvalues in the proof of Corollary~\ref{cor:poly}. 

Lemma~\ref{l:tauSB} also shows that there exists $C>0$ such that for arbitrary fixed $\xi \in (0,1)$ and $\eps \in (0,\xi)$,
\begin{align*}
 -\log\Pi_\tau(\|f\| \leq \xi\eps)\leq -2^{d+1} C^{2}\log \Pi_\tau(\|f\| \leq \eps),
\end{align*}
verifying condition \eqref{e:prior.reg.sc} for $\tau\in(0,1)$.

Condition \eqref{e:hyper.mass} is also satisfied because the hyper-prior is uniform on a set with cardinality of the order $\log n$.  

Again the proof is concluded by bounding the KL-divergence between the variational posterior and the true posterior. In this case we note that for
\begin{align*}
	\tau & = \lambda_n \asymp n^{-1/(d+2\beta)} (\log n)^{1+1/(d+2\beta)}, \\
	m & = m_{n,\tau} \geq \tau^{-d} (\log n)^{d+1}
\end{align*}
we have
\begin{align*} 
	\bb E_{\bs x} \tr(K_\ff-Q_\ff)
	& \lesssim n \sum_{j>m} (s_j^\tau)^2 
	 \lesssim n \int_m^\infty \tau^d e^{-\tau x^{1/d}} \,\mathrm dx \\
	& \lesssim n \int_{m^{1/d}}^\infty \tau^d d u^{d-1} e^{-\tau u} \,\mathrm du 
	 \lesssim n (\tau m^{1/d})^{d-1} e^{-\tau m^{1/d}}.
\end{align*}
Since $x \mapsto x^{d-1} e^{-x}$ is decreasing on $(d-1,\infty)$ it follows that
\[
	\bb E_{\bs x} \|K_\ff-Q_\ff\|_{\mathrm{op}} \leq \bb E_{\bs x} \tr(K_\ff-Q_\ff)
	\lesssim n (\log n)^{(1+1/d)(d-1)} e^{-(\log n)^{1+1/d}}.
\]
Consequently, both the trace term and operator norm term decay faster than any power of $n$, hence Lemma 3 in \cite{vbgp} gives
\[
	\bb E_{f_0} \KL(\tilde\Pi_{\lambda_n}(\,\cdot\,|\bs x,\bs y) \|
	\Pi_{\lambda_n}(\,\cdot\,|\bs x,\bs y)) \lesssim n\overline\eps_n^2(\lambda_n) \asymp n(n/\log n)^{-2\beta/(d+2\beta)}.
\]
We slightly enlarge $\delta_n$ to $c_n(n/\log n)^{-\beta/(d+2\beta)}$. Then all conditions in Theorem~\ref{t:vbp.contr} are (still) satisfied and $\eps_{n,0} \lesssim \delta_n$, concluding the proof. 

~

\subsection{Proof of Corollary \ref{cor:dim}}\label{sec:cor:dim}

We apply Theorem~\ref{t:vbp.contr}. In view of Lemma \ref{l:dimRKHS} and \ref{l:dimSB}, for $C>0$ large enough
\[
	\overline \eps_n(D) = C (\sqrt{D n^{-1} \log n} \vee D^{-\beta/d})
\]
solves the concentration inequality $\varphi_{f_0,\lambda}(\overline\eps_n(D)) \leq n\overline\eps_n^2(D)$ for $n$ large enough. This provides an upper bound for the oracle rate $\eps_n(D)$. The value of $D$ that minimises the upper bound is of the order $(n/\log n)^{d/(d+2\beta)}$ and hence
\[
	\min_{D\in\Lambda_n} \eps_n(D) \lesssim (n/\log n)^{-\beta/(d+2\beta)}.
\]

The lower bound on the small ball exponent in Lemma~\ref{l:dimSB} gives a lower bound for $\eps_n(D)$. Solving $\varphi_{0,D}(\underline\eps_n(D)) \geq n\underline\eps_n^2(D)$ it follows that $\eps_n(D) \gtrsim (D n^{-1} \log n)^{1/2}$, hence $\min_{D\in\Lambda_n} n\eps_n^2(D) \to \infty$. 

Next, for $J_0 = n/(\log n)^2$
\[
	\|f_0^{>J_0}\|_\infty
	\leq \sum_{j>J_0} j^{-2\beta/d} j^{2\beta/d} |\ip{f_0,\varphi_j}| \\
	= o(J_0^{-2\beta/d}) = o(n^{-1})
\]
so condition \eqref{e:f0.tail} is satisfied for $\delta_n = n^{-\beta/(d+2\beta)}$. Since $J_0 > D$ for any $D \in \Lambda_n$, the right-hand side of \eqref{e:priortail} is always zero, so this condition is trivially satisfied. 

By Lemma~\ref{l:dimSB} with $C,\eps_0>0$ as given there, it follows that for $\eps \in (0,\xi\wedge\eps_0)$
\[ -\log\Pi_D(\|f\| \leq \xi \eps) \leq CD \log \frac{1}{\eps^2} \leq -2C^2\log\Pi_D(\|f\| \leq \eps), \]
which verifies the condition \eqref{e:prior.reg.sc}. 

Since $\pi$ is uniform it follows that
\[
	-\log\pi(D) = \log \lfloor \sqrt n \rfloor \leq \frac 12 \log n
\]
which shows \eqref{e:hyper.mass} holds. 

It remains to bound the KL-divergence. For $D$ as above, take the distribution $Q = \mathcal N(\mu,\Sigma)$, where $\Sigma \in \bb R^{D\times D}$ a diagonal matrix with entries $(D + n/\sigma^2)^{-1}$ and $\mu = \sigma^{-2} \Sigma\Phi^\T \bs y$. For convenience let us write $\hat\Sigma = (DI + \sigma^{-2}\Phi^\T \Phi)^{-1}$ and $\hat\mu = \sigma^{-2}\hat\Sigma\Phi^\T\bs y$. Using the exact form of $\Pi_D(\,\cdot\,|\bs x,\bs y)$ it can be seen that
\begin{equation}\label{e:kl.dim}
	\KL(Q \| \Pi_D(\,\cdot\,|\bs x,\bs y))
	= \frac 12 \Big( \log \frac{|\hat\Sigma|}{|\Sigma|} - D
	    + \tr(\Sigma\hat\Sigma^{-1}) + (\mu-\hat\mu)^\T\hat\Sigma^{-1}(\mu-\hat\mu)\Big).
\end{equation}
In expectation the trace term cancels out $-D$ so we have to bound the expectation of the logarithmic term and the quadratic form in this display. 

Akin to \eqref{e:hoef.induc}, by Hoeffding's inequality, note that 
\[
	\bb P_{f_0}^{(n)} \Big(\max_{j,\ell \leq D} \Big|\sum_{i=1}^n \varphi_\ell(x_i)\varphi_j(x_i)
	- n\delta_{j\ell}\Big|\geq C \sqrt{n\log n}\Big)\leq 2D^2n^{-C^2/(2C_\varphi^2)}.
\]
By increasing $C>0$ the right hand side is bounded by an arbitrary negative power of $n$. Let us define the complementary event
\[
	\mathcal H = \Big\{ \max_{j,\ell \leq D}  \Big|\sum_{i=1}^n \varphi_\ell(x_i)\varphi_j(x_i)
	- n\delta_{j\ell}\Big| < C\sqrt{n\log n} \Big\}.
\]
Now by Gershgorin's Circle Theorem, on the event $\mathcal H$, any eigenvalue $\nu$ of the matrix $\Phi^\T\Phi$ satisfies for some $j$ between $1$ and $D$
\begin{align*}
	|\nu - n|
	&\leq \Big|\nu - \sum_{i=1}^n \varphi_j(x_i)^2\Big|
	+ \Big|\sum_{i=1}^n \varphi_j(x_i)^2 - n\Big| \\
	&\leq \sum_{\ell \neq j}\Big|\sum_{i=1}^n \varphi_\ell(x_i)\varphi_j(x_i)\Big|
	+ \Big|\sum_{i=1}^n \varphi_j(x_i)^2 - n\Big| \\
	&=\sum_{\ell=1}^D \Big|\sum_{i=1}^n \varphi_\ell(x_i)\varphi_j(x_i)-n\delta_{\ell j}\Big| \\
	& \leq C D \sqrt{n \log n}.
\end{align*}
For the optimal $D$ of the order $(n/\log n)^{d/(d+2\beta)}$, it follows from the assumption $\beta >d/2$ that the upper bound is $o(n)$ and 
\[
	\bs 1_{\mathcal H} \log\frac{|\hat\Sigma|}{|\Sigma|} 
	\leq D \log \frac{D+\sigma^{-2}n}{D+\sigma^{-2}n - CD\sqrt{n\log n}}
	= o(D) = o(n\eps_n^2(D)).
\]

To bound the quadratic term in \eqref{e:kl.dim} we need a slightly tighter bound on the largest eigenvalue of the difference
\[
	\tilde\Sigma^{-1} - \Sigma^{-1} = \Phi^\T\Phi - n I = \sum_{i=1}^n (\varphi_{1:D}(x_i)\varphi_{1:D}(x_i)^\T - I)
\]
for the same optimal value of $D$. We bound the largest eigenvalue of this difference by the Bernstein inequality for matrices. By e.g. Theorem 1.6 in \cite{tropp2012},
\begin{equation}\label{e:bernstein}
	\bb P_{f_0}^{(n)} (\| \Phi^T\Phi-nI \|_{\mathrm{op}} \geq t)
	\leq D\exp\Big(-\frac{t^2/2}{S+Rt/3}\Big)
\end{equation}
where $R$ is a deterministic upper bound for the largest absolute eigenvalue of the matrix $\varphi_{1:D}(x_i)\varphi_{1:D}(x_i)^\T - I_{\mathrm{op}}$ and
\begin{align*}
	S &=\Big\| \sum_{i=1}^n \bb E_{f_0} (\varphi_{1:D}(x_i)\varphi_{1:D}(x_i)^\T - I)^2 \Big \|_{\mathrm{op}} \\
	& \leq n \Big\| \bb E_{f_0}\Big(\varphi_{1:D}(x_i)[\varphi_{1:D}(x_i)^\T\varphi_{1:D}(x_i)]\varphi_{1:D}(x_i)^\T\Big)\Big\|_{\mathrm{op}} \\
	& \leq n C_\varphi^2 D \Big\| \bb E_{f_0} \varphi_{1:D}(x_i)\varphi_{1:D}(x_i)^\T \Big\|_{\mathrm{op}} \\
	& = C_\varphi^2 D n
\end{align*}
as follows from the inequality
\[
\varphi_{1:D}(x_i)^\T\varphi_{1:D}(x_i) = \sum_{j=1}^D \varphi_j(x_i)^2 \leq C_\varphi^2 D.
\]
The above display also implies that we can let $R = C_\varphi^2 D$. For $t = C \sqrt{D n \log n}$ it follows that the term $Rt$ is dominated by the upper bound $C_\varphi^2 Dn$ for $S$, so for $n$ large enough
\[
	-\frac{t^2/2}{S+Rt/3} \leq -\frac{C^2 D n\log n}{4C_\varphi^2 Dn} = - \frac{C^2}{2C_\varphi^2} \log n. 
\]
By increasing the constant $C$, the probability in \eqref{e:bernstein} is bounded by an arbitrary negative power of $n$, hence the complementary event
\[
	\mathcal B := \Big\{ \|\Phi^T\Phi - n I\|_{\mathrm{op}} \leq C \sqrt{Dn\log n} \Big\}
\]
has probability tending to $1$. 

We are finally ready to bound the term
\[
	(\mu-\hat\mu)^\T\hat\Sigma^{-1}(\mu-\hat\mu)
	= \sigma^{-4}\bs y^\T\Phi(\hat\Sigma-\Sigma)\hat\Sigma^{-1}(\hat\Sigma-\Sigma)\Phi^\T\bs y.
\]
The Hoeffding bound shows that on the event $\mathcal H$ the eigenvalues of $\hat\Sigma^{-1} = DI + \sigma^{-2} \Phi^\T\Phi$ are bounded by $D + \sigma^{-2} n \pm C D \sqrt{n \log n} \asymp n$. Consequently, combined with the Bernstein bound for $\tilde\Sigma^{-1} - \Sigma^{-1} = \Phi^\T\Phi - n I$, we obtain
\begin{align*}
	\|(\hat\Sigma-\Sigma)\hat\Sigma^{-1}(\hat\Sigma-\Sigma)\|
	& = \|\Sigma(\hat\Sigma^{-1}-\Sigma^{-1})\hat\Sigma(\hat\Sigma^{-1}-\Sigma^{-1})\Sigma\| \\
	& \lesssim \frac{1}{n^3} D n \log n ,
\end{align*}
and it follows that
\begin{equation}\label{e:kl.quad}
	\bb E_{f_0} (\mu-\hat\mu)^\T\hat\Sigma^{-1}(\mu-\hat\mu)\bs 1_{\mathcal H \cap \mathcal B} \lesssim \frac{D \log n}{n^2} \bb E_{f_0} \bs y^\T\Phi\Phi^\T \bs y.
\end{equation}
The term on the right can be expanded as
\[
	\bb E_{f_0} \bs y^\T \Phi\Phi^\T \bs y
	=\sum_{i=1}^n \sum_{j=1}^n \sum_{\ell=1}^D
	\bb E_{f_0} y_i \varphi_\ell(x_i)\varphi_\ell(x_j) y_j.
\]
For $i \neq j$ the terms in the sum equal $\ip{f_0,\varphi_\ell}^2$. For $i=j$, we note that by the Cauchy-Schwarz inequality,
\begin{align*}
	|f_0(x)|^2
	& = \Big(\sum_{j=1}^\infty j^{\beta/d} \ip{f_0,\varphi_j} j^{-\beta/d} \varphi_j(x)\Big)^2 \\
	& \leq \Big(\sum_{j=1}^\infty j^{2\beta/d} \ip{f_0,\varphi_j}^2\Big)\Big(\sum_{j=1}^\infty C_\varphi^2 j^{-2\beta/d}\Big)
\end{align*}
so $f_0$ is bounded due to the assumptions on its Sobolev smoothness and the uniform bound on $\varphi_j$. This implies that for $i=j$ 
\[
	\bb E_{f_0} y_i \varphi_\ell(x_i)\varphi_\ell(x_j) y_j
	= \sigma^2 \bb E_{f_0}\varphi_\ell(x_i)^2 + \bb E_{f_0} f_0(x_i)^2 \varphi_j(x_i)^2
	\lesssim 1.
\]
Altogether we obtain
\[
	\bb E_{f_0} \bs y^\T \Phi\Phi^\T \bs y \lesssim n^2,
\]
so recalling \eqref{e:kl.quad} it follows that
\begin{align*}
	\bb E_{f_0} (\mu-\hat\mu)^\T\hat\Sigma^{-1}(\mu-\hat\mu)\bs 1_{\mathcal H \cap \mathcal B} 
	& \lesssim D \log n \\
	& \lesssim n^{d/(d+2\beta)} (\log n)^{-2\beta/(d+2\beta)} \\
	& = n\eps_n^2(D),
\end{align*}
completing the proof of the corollary.

\section{Technical lemmas}\label{sec:tech:lem}
The following lemma is a standard testing result using empirical $L_2$-norm. Below we provide a version with explicit constants and for completeness provide a proof as well. 

\begin{lemma}\label{l:emp.test}
For any $f_0,f_1 \in L^2(\mathcal X, G)$, there exists a sequence of tests $\phi_n$ satisfying
\begin{align*}
\bb E_{f_0}[\phi_n|\bs x]
&\leq \exp(-n\|f_1-f_0\|_n^2/(8\sigma^2)) \\
\sup_{f : \|f-f_1\|_n \leq \|f_1-f_0\|_n/3} \bb E_f [1-\phi_n|\bs x] &\leq \exp(-n\|f_1-f_0\|_n^2/(72\sigma^2)).
\end{align*}
\end{lemma}
\begin{proof}
By the Chernoff bound, using that under $\bb P_{f_0}^{(n)}(\,\cdot\,|\bs x)$,
\[
\sum_{i=1}^n (y_i-f_0(x_i))(f_1(x_i)-f_0(x_i)) \sim \mathcal N\Big(0, \sigma^2 \sum_{i=1}^n(f_1(x_i)-f_0(x_i))^2 \Big),
\]
the likelihood ratio test $\phi_n = \bs 1\{ p^{(n)}_{f_1}(\bs x,\bs y) \geq p^{(n)}_{f_0}(\bs x,\bs y) \}$ has type 1 error
\begin{align*}
\bb P_{f_0}^{(n)}[\phi_n|\bs x]
& = \bb P_{f_0}^{(n)}\Big( \sum_{i=1}^n (y_i-f_1(x_i))^2 \leq \sum_{i=1}^n (y_i-f_0(x_i))^2 \Big| \bs x\Big) \\
& = \bb P_{f_0}^{(n)}\Big( 2\sum_{i=1}^n(y_i-f_0(x_i))(f_1(x_i)-f_0(x_i))\geq \sum_{i=1}^n(f_1(x_i)-f_0(x_i))^2 \Big| \bs x\Big) \\
& \leq \exp\Big(-\frac{(\sum (f_1(x_i)-f_0(x_i))^2)^2}{8\sigma^2\sum(f_1(x_i)-f_0(x_i))^2}\Big)  \\
& = \exp(-n\|f_1-f_0\|_n^2/(8\sigma^2)).
\end{align*}
For $f$ such that $\|f-f_1\|_n \leq \|f_1-f_0\|_n/3$, note that
\[
\|f_0-f\|_n \geq \|f_0-f_1\|_n - \|f_1-f\|_n \geq \frac 23 \|f_0-f_1\|_n,
\]
so, similarly, the type 2 error of this test is
\begin{align*}
\bb P_f^{(n)}[1-\phi_n|\bs x]f
&= \bb P_f^{(n)}\Big( \sum_{i=1}^n (y_i-f_1(x_i))^2 \geq \sum_{i=1}^n (y_i-f_0(x_i))^2 \Big| \bs x\Big) \\
& = \bb P_f^{(n)}\Big( 2\sum_{i=1}^n(y_i-f(x_i))(f_0(x_i)-f_1(x_i)) \geq n \|f_0-f\|_n^2 - n \|f_1-f\|_n^2 \Big| \bs x\Big) \\
& \leq \bb P_f^{(n)}\Big( 2\sum_{i=1}^n(y_i-f(x_i))(f_0(x_i)-f_1(x_i)) \geq n \|f_0-f_1\|_n^2/3 \Big| \bs x\Big) \\
& \leq \exp(-n\|f_1-f_0\|_n^2/(72\sigma^2)).  
\end{align*}
\end{proof}

\begin{lemma}\label{l:tauRKHS}
Let $\bb H_\tau$ be the RKHS of the prior in \eqref{e:expoPrior}. If $f_0 \in \mathcal S^\beta(L)$ for some $\beta>0$ and $L>0$, then for all $\tau > 0$ and any sufficiently small $\eps>0$,
\[ \inf_{\substack{h \in \bb H_\tau : \\ \|h-f_0\| \leq \eps}} \|h\|_{\bb H_\tau}^2 \leq \tau^{-d}\Big(L^2 e^\tau \vee \eps^2 \exp(2\tau(L/\eps)^{1/\beta})\Big). \]
\end{lemma}

\begin{proof}
For $\eps$ sufficiently small, let $J$ be a positive integer such that
\[
	\eps J^{\beta/d} \geq L > \eps (J-1)^{\beta/d},
\]
and define the function $h = \sum_{j \leq J} \ip{f_0,\varphi_j}\varphi_j \in \bb H_\tau$. Note that for this $J$ and $h$,
\[
	\|h-f_0\|^2 = \sum_{j>J} \ip{f_0,\varphi_j}^2 \leq J^{-2\beta/d} L^2 \leq \eps^2.
\]
By convexity of the function $x \mapsto e^{\tau x} x^{-2\beta}$ for positive $x$, it follows that
\[
	\max_{1\leq j \leq J}e^{\tau j^{1/d}}j^{-2\beta/d} = e^\tau \vee e^{\tau J^{1/d}} J^{-2\beta/d}.
\] 
Then we obtain
\begin{align*}
	\|h\|_{\bb H_\tau}^2
	& = \sum_{j \leq J} \tau^{-d}e^{\tau j^{1/d}} j^{-2\beta/d} \cdot
	j^{2\beta/d} \ip{f_0,\varphi_j}^2 \\
	& \leq \tau^{-d} (e^\tau \vee e^{\tau J^{1/d}} J^{-2\beta/d}) L^2 \\
	& \leq \tau^{-d}(L^2 e^\tau \vee \eps^2 \exp(2\tau(J-1)^{1/d}))
\end{align*}
so the statement of the lemma follows.
\end{proof}

\begin{lemma}\label{l:tauSB}
Let $\Pi_\tau$ be the prior in \eqref{e:expoPrior}. There exists a constant $C>0$ such that for all $\tau \in (0,1]$ and small enough $\eps > 0$,
\[
C^{-1}\tau^{-d} (\log \tfrac1\eps)^{d+1} \leq -\log\Pi_\tau(\|f\| \leq \eps) \leq C \tau^{-d} (\log \tfrac1\eps)^{d+1}.  
\]
\end{lemma}
\begin{proof}
For the upper bound, we follow the proof of Lemma 11.47 in \cite{ghosal2017}. For all sufficiently small $\eps$, let $J \in \bb N$ be such that
\[
	\tau^d J e^{-\tau J^{1/d}} \leq \eps^2/2 \leq \tau^d (J-1) e^{-\tau (J-1)^{1/d}}.
\]
We note that such $J$ exists, since the function $x\rightarrow x^de^{-x}$ is monotone decreasing for $x\geq d$ and the largest solution $x^*$ of $x^de^{-x}=\eps^2/2$ for $\eps\rightarrow 0$ tends to infinity. Hence  $\tau J^{1/d} \asymp x^* \to \infty$ and for  $\tau (J-1)^{1/d}$ the function will take a larger value than $\eps^2/2$. Furthermore, note that $\sum_{j \leq J} j^{1/d} \geq \int_0^J x^{1/d} \,\mathrm dx = \frac{d}{d+1}J^{(d+1)/d}$, so
\begin{align*}
\Pi_\tau\Big( \sum_{j \leq J} \ip{f,\varphi_j}^2 \leq \eps^2/2 \Big) 
& \geq \Pr\Big( \sum_{j \leq J} \tau^d e^{-\tau J^{1/d}} Z_j^2 \leq \eps^2/2 \Big) \prod_{j \leq J} \frac{e^{-\tau J^{1/d}/2}}{e^{-\tau j^{1/d}/2}}  \\
& \geq \Pr\Big(\sum_{j \leq J} Z_j^2 \leq J \Big) \exp\Big(-\tau \frac{J^{(d+1)/d}}{d+1} \Big) .
\end{align*}
Note that $J \geq \tau^d J \to \infty$ as $\eps\to 0$, and $\Pr(\sum_{j \leq J} Z_j^2 \leq J) \to 1/2$ using the Central Limit Theorem. Regarding the remaining part of the sum, by Markov's inequality,
\begin{align*}
\Pi_\tau\Big( \sum_{j > J} f_j^2 \leq \eps^2/2 \Big) 
& \geq 1 - 2\eps^{-2} \sum_{j > J} \tau^d e^{-\tau j^{1/d}} \\
& \geq 1 - 2\eps^{-2} \int_J^\infty \tau^d e^{-\tau x^{1/d}} \,\mathrm dx \\
& = 1 - 2\eps^{-2} \sum_{i=0}^{d-1} \frac{d!}{i!} \tau^i J^{i/d} e^{-\tau J^{1/d}} \\
& \geq 1 -  \sum_{i=0}^{d-1} \frac{d!}{i!} (\tau J^{1/d})^{i-d} \to 1,
\end{align*}
where the identity for the integral follows by repeated partial integration. Combining all of the above and realising that $\log \frac1\eps \asymp \tau J^{1/d}$ as $\eps \to 0$, it follows that
\[
-\log \Pi_\tau(\|f\|^2 \leq \eps^2) \leq \tau \frac{J^{(d+1)/d}}{d+1} + \log 2 + o(1) \lesssim (\log \tfrac1\eps)^{d+1} \tau^{-d}.
\]

To bound the small ball probability from below, fix $\delta > 1$ and (for $\eps>0$ sufficiently small) take $J$ to be the largest positive integer satisfying the inequality $J > \eps^2 \tau^{-d} e^{\tau J^{1/d}}$. The Chernoff bound for the $\chi_J^2$ distribution yields
\begin{align*}
\Pi_\tau(\|f\|\leq\eps)
& \leq \Pr \Big(\sum_{j \leq J} Z_j^2 \leq \eps^2 \tau^{-d} e^{\tau J^{1/d}} \Big) \\
& \leq \exp\Big(J/2 - \eps^2 \tau^{-d} e^{\tau J^{1/d}}/2 + (J/2)\log(\eps^2 \tau^{-d} e^{\tau J^{1/d}}/J) \Big) \\
& \leq \exp\Big((1+(1-\delta)\tau J^{1/d})J/2\Big),
\end{align*}
so, using again $\log \frac1\eps \asymp \tau J^{1/d}$,
\[ - \log \Pi_\tau(\|f\| \leq \eps) \geq -J/2 + \frac{\delta-1}{2} \tau J^{(d+1)/d} \gtrsim \tau J^{(d+1)/d} \gtrsim \tau^{-d} (\log \tfrac1\eps)^{d+1}.  \]
\end{proof}

\begin{lemma}\label{l:dimRKHS}
Let $\Pi_D$ be the prior defined in \eqref{e:dimprior} and $\bb H_D$ its RKHS. Then for any $\eps>0$, $f_0 \in \mathcal S^\beta(L)$ and $D \geq (L/\eps)^{d/\beta}$, 
\[ \inf_{\substack{h \in \bb H_\lambda : \\ \|h-f_0\| \leq \eps}} \|h\|_{\bb H_D}^2 \leq D \|f_0\|^2. \]
\end{lemma}
\begin{proof}
The argument is as in the proof of Lemma~\ref{l:tauRKHS}. Taking $h = \sum_{j=1}^D \ip{f_0,\varphi_j}\varphi_j$ we note that $\|h-f_0\|^2 \leq D^{-2\beta/d} L^2$ and so the infimum is bounded by $\|h\|_{\bb H_D}^2 = D \sum_{j\leq D} \ip{f_0,\varphi_j}^2$. 
\end{proof}

\begin{lemma}\label{l:dimSB}
Let $\bb H_\lambda$ be the RKHS of the prior \eqref{e:dimprior}. There exist constants $\eps_0>0$ and $C>0$ such that for all $\eps \in (0,\eps_0)$ and $\lambda \in (0,1)$,
\[ C^{-1} D \log \frac 1\eps \leq -\log \Pi_\lambda(\|f\| \leq \eps) \leq C D \log \frac 1\eps. \]
\end{lemma}
\begin{proof}
For $X_D^2$ a chi-squared random variable with $D$ degrees of freedom and $t = (\eps^{-2}-1)/2 > 0$, the Chernoff bound gives
\begin{align*}
\Pi_\lambda(\|f\| \leq \eps)
& = \Pr \Big(D^{-1} \sum_{j=1}^D Z_j^2 \leq \eps^2\Big) \\
& \leq e^{tD \eps^2} \bb E e^{-t X_D^2} \\
& = \exp(D/2 - D\eps^2/2 + D \log \eps).
\end{align*}
On the other hand, note first that by Stirling's approximation for the Gamma function,
\[
\Gamma(D/2+1) \lesssim \sqrt{\pi D} \Big(\frac{D}{2e}\Big)^{D/2},
\]
so by partial integration,
\begin{align*}
\Pi_\lambda(\|f\| \leq \eps)
& = \int_0^{D\eps^2} \frac{x^{D/2-1} e^{-x/2}}{2^{D/2}\Gamma(D/2)} \,dx \\
& = \frac{(D\eps^2)^{D/2} e^{-D\eps^2/2}}{2^{D/2}\Gamma(D/2+1)} + \int_0^{D\eps^2} \frac{x^{D/2} e^{-x/2}}{2^{D/2+1}\Gamma(D/2+1)} \,dx \\
& \geq \frac{\eps^D e^{-D\eps^2/2}}{(2/D)^{D/2} \Gamma(D/2+1)} \\
& \gtrsim \exp\Big(D \log \eps - D\eps^2/2 - \tfrac 12 \log (\pi D) + D/2\Big).
\end{align*}
The result follows by realising that the term $D \log \eps$ dominates for small $\epsilon$.
\end{proof}

\bibliography{references2025}

\end{document}